\documentclass[12pt]{article}
\usepackage[T2A]{fontenc}
\usepackage[cp866]{inputenc}
\usepackage[tbtags]{amsmath}
\usepackage{amsfonts,amssymb,mathrsfs,amscd, amsthm,comment}
\usepackage{datetime}
\usepackage{mathrsfs}

\usepackage[breaklinks]{hyperref}

\hypersetup{
unicode=true,
colorlinks=true,
linkcolor=blue,
citecolor=blue,
urlcolor=blue,
filecolor=blue,
bookmarksnumbered=true,
pdfstartview=FitH,
pdfhighlight=/N
}

\newdimen\normalparindent
\newtheorem{theorem}{Theorem}
\newtheorem{Corollary}{Corollary}
\newtheorem{Proposition}{Proposition}
\newtheorem{Statement}{Statement}
\newtheorem{Lemma}{Lemma}

\newtheorem{Example}{Example}
\newtheorem{Remark}{Remark}

\def\SR{\mathcal S}
\def\RS{\mathfrak R}
\def\TRS{\widetilde{\mathfrak R}}
\def\SR{\mathcal S}
\def\OO{\mathbf O}
\def\zz{\mathbf z}
\def\aa{\mathbf a}
\def\pp{\mathbf p}
\def\qq{\mathbf q}
\def\bb{\mathbf b}
\def\ww{\mathbf w}
\let\myh\widehat\let
\myt\widetilde\let\myo\overline
\def\tpi{\widetilde{\pi}}
\def\CC{\mathbb C}
\def\NN{\mathbb N}
\def\MM{\mathscr M}
\def\OO{\mathbf O}

\def\DD{\mathbf D}
\def\ddc{\operatorname{dd^c}}
\def\dsh{\operatorname{\delta-sh}}

\def\Id{\operatorname{Id}}

\def\mcap{\operatorname{cap}}
\def\sp{\operatorname{sp}}
\def\rhos{{\rho_{\sp}}}
\def\dist{\operatorname{dist}}
\def\mDelta{\operatornamewithlimits{\Delta}}
\def\loc{\operatorname{loc}}
\def\Pot{\operatorname{Pot}}
\def\const{\operatorname{const}}
\def\Meas{\operatorname{Meas}}

\begin{document}

%\selectlanguage{english}

\title{\bf{
Polynomial Hermite--Pad\'e $m$-system for meromorphic functions on a~compact Riemann surface}}

\author{A.~V.~Komlov}
%\address{Steklov Mathematical Institute of Russian Academy of Sciences, Moscow, Russia}
%\email{komlov@mi.ras.ru}

\date{}
%\udk{}

\maketitle

\markright{Polynomial Hermite--Pad\'e $m$-system}

%\begin{fulltext}

\begin{abstract}
For an arbitrary tuple of   $m+1$ germs of analytic functions at a~fixed point, we introduce the so-called polynomial \textit{Hermite--Pad\'e $m$-system} (of order~$n$,
$n\in\mathbb N$), which consists of $m$~tuples of polynomials; these tuples, which are indexed by
a~natural number $k\in[1,\dots,m]$, are called the  \textit{$k$th polynomials of the Hermite--Pad\'e $m$-system}.
We study the weak asymptotics of the polynomials of the Hermite--Pad\'e $m$-system  constructed at the point~$\infty$  from the
tuple of germs $[1, f_{1,\infty},\dotsc$, $f_{m,\infty}]$ of the functions $1, f_1,\dots,f_m$ that
are meromorphic  on some $(m+1)$-sheeted branched covering $\pi\colon \RS\to\myh\CC$
of the Riemann sphere $\myh\CC$ of a~compact Riemann surface $\RS$. In particular, under some additional condition on~$\pi$,
we find the limit distribution of the zeros and the asymptotics of the ratios of the $k$th polynomials for all $k\in[1,\dots, m]$.
It turns out that in the case, where $f_j = f^j$ for some meromorphic function~$f$ on~$\RS$,
the ratios of some $k$th polynomials of such Hermite--Pad\'e $m$-system converge to the sum of the values of the
function~$f$ on the first $k$~sheets of the Nuttall partition of the Riemann surface $\RS$ into sheets.
\end{abstract}

\setcounter{tocdepth}{1}\tableofcontents

\makeatletter
\renewcommand{\@makefnmark}{}
\makeatother
\footnotetext{This work was supported by the Russian Science Foundation under grant 19-11-00316.}

\section{Introduction}
\label{s1}

Assume first that  $f_{0, \infty}(z)\equiv 1,  f_{1, \infty}(z), \dots, f_{m, \infty}(z)$ are  $m+1$
arbitrary analytic germs at infinity  (on the Riemann sphere~$\myh\CC$).
We fix a~natural number $k\in\{1,\dots,m\}$ and, for each $n\in\mathbb N$, define
the tuple of  ${m+1}\choose{k}$ `\textit{$k$th polynomials of the Hermite--Pad\'e $m$-system}' of order~$n$,
which are constructed from the tuple of germs  $[1, f_{1, \infty}, \dots,
f_{m, \infty}]$ at the point $\infty\in\myh\CC$ as follows.  These are the polynomials
$P_{n;i_1,\dots,i_k}$, $0\le i_1<i_2<\dots<i_k\le m$, such that
$\deg P_{n;i_1,\dots,i_k}\le (m+1-k)n$, at least one
$P_{n;i_1,\dots,i_k}\not\equiv 0$, and, for each index set $0<j_1<\dots<j_k\le m$,
\begin{equation}
\label{khp_i}
P_{n;j_1\dots,j_k}(z) +
\sum_{s=1}^k (-1)^{s}P_{n;0,j_1,\dots,j_{s-1}, j_{s+1},\dots,j_k}(z)f_{j_s,\infty}(z)=O\left(\frac{1}{z^{kn+1}}\right)
\end{equation}
as $z\to\infty$.
It is easily seen that condition \eqref{khp_i} is a~system of $(n(m+1)+1) {m\choose k}$ linear homogeneous equations
for $(n(m+1-k)+1){m+1\choose k}=n(m+1){m\choose k}+{m+1\choose k}$ unknown coefficients of the polynomials $P_{n;i_1,\dots,i_k}$.
The coefficients of this system are some linear expressions of the first $(m+1)n$ Taylor coefficients of the  germs $f_{s, \infty}$. Hence the
polynomials $P_{n;i_1,\dots,i_k}$ always exist, but, in general, are not unique.
(Note that the Hermite--Pad\'e $m$-system can be constructed also from the germs $f_{j, \infty}$ which
are meromorphic at~$\infty$. In this case, one should increase the maximum possible degree
of the polynomials $P_{n;i_1,\dots,i_k}$. For example, it suffices to require that $\deg P_{n;i_1,\dots,i_k}\le (m+1-k)n+M$, where $M$
is the maximal order of the poles of the germs $f_{j, \infty}$ at~$\infty$.)

It is clear that  conditions \eqref{khp_i} are linearly independent; however, the index~0 in them `plays a~special role'.
One can give a~different, the so-called homogeneous definition of the $k$th polynomials of the Hermite--Pad\'e $m$-system,
in which the conditions are no longer linearly independent: $\deg P_{n;i_1,\dots,i_k}\le (m+1-k)n$, at least one
$P_{n;i_1,\dots,i_k}\not\equiv 0$, and, for each index set $0\le i_0<i_1<\dots<i_k\le m$,
\begin{equation}
\label{khp}
\sum_{s=0}^k (-1)^{s}P_{n;i_0,\dots,i_{s-1}, i_{s+1},\dots,i_k}(z)f_{i_s,\infty}(z)=O\left(\frac{1}{z^{kn+1}}\right)
\quad \text{as } z\to\infty.
\end{equation}
It is clear that all conditions  \eqref{khp_i} are contained in \eqref{khp} (for $[i_0, i_1,\dots,i_k]:=[0, j_1,\dots,j_k]$).
In order to verify that conditions \eqref{khp_i} are sufficient for \eqref{khp} to hold with  $i_0\ne0$,
it suffices to substitute in \eqref{khp} the expressions for all $P_{n;i_0,\dots,i_{s-1}, i_{s+1},\dots,i_k}$,
which are obtained from \eqref{khp_i} for $[j_1,\dots,j_k]:=[i_0,\dots,i_{s-1}, i_{s+1},\dots,i_k]$, and
check that the right-hand side of \eqref{khp} is~$0$ up to  $O(z^{-(kn+1)})$.
So, definitions \eqref{khp_i} and \eqref{khp} are equivalent. In what follows, we will mostly use the first one.

Let us now recall the definitions of the classical Hermite--Pad\'e polynomials  of type~I and~II.
The Hermite--Pad\'e polynomials  of type~I of order~$n$ constructed from the tuple of germs $[1, f_{1, \infty}, \dots,
f_{m, \infty}]$ at the point $\infty\in\myh\CC$ are the polynomials $Q_{n, i}$, $0\le i \le m$,
such that $\deg Q_{n,i}\le n$, at least one  $Q_{n, i}\not\equiv 0$, and in the neighbourhood of infinity
\begin{equation}
\label{hp1}
\sum_{j=0}^m Q_{n,j}(z)f_{j,\infty}(z)=O\left(\frac{1}{z^{m(n+1)}}\right) \quad \text{as } z\to\infty.
\end{equation}
The Hermite--Pad\'e polynomials  of type~II of order~$n$ constructed from the tuple of germs $[1, f_{1, \infty}, \dots,
f_{m, \infty}]$ at the point $\infty\in\myh\CC$ are the polynomials $q_{n, i}$, $0\le i \le m$,
such that $\deg q_{n,i}\le mn$, at least one  $q_{n, i}\not\equiv 0$, and, for all $1\le j \le m$,
\begin{equation}
\label{hp2}
q_{n, 0}(z)f_{j,\infty}(z) - q_{n, j}(z) = O\left(\frac{1}{z^{n+1}}\right)\quad \text{as } z\to\infty
\end{equation}
in the neighbourhood of infinity.

Note that  conditions ~\eqref{khp_i}, which define the $1$st polynomials of the Hermite--Pad\'e $m$-system  (that is, for $k=1$),
coincide  (up to a~sign) with conditions \eqref{hp2}, which define the Hermite--Pad\'e polynomials  of type~II.
So, the  $1$st polynomials of the Hermite--Pad\'e $m$-system  are precisely the Hermite--Pad\'e polynomials of
type~II, that is,  $P_{n;i}\equiv q_{n,i}$. At the same time, condition~\eqref{khp_i},
which defines the $m$th polynomials of the Hermite--Pad\'e $m$-system  (that is for $k=m$)
is practically equal to condition \eqref{hp1}, which defines the Hermite--Pad\'e polynomials of type~I.
More precisely, if one defines  $P_{n;0,1,\dots,j-1,j+1,\dots,m}:=(-1)^jQ_{n;j}$, then the left-hand side of \eqref{khp_i} for $k=m$ and \eqref{hp1}
will be equal and the order of contact of~0 in their right-hand sides will be  $m(n+1)$ and $mn+1$, respectively.
So, the polynomials $(-1)^jQ_{n;j}$ automatically satisfy
condition~\eqref{khp_i}, that is,  the Hermite--Pad\'e polynomials of type~I
(with altered sign of the odd polynomials) are the particular case of the  $m$th polynomials of the
Hermite--Pad\'e $m$-system.
Moreover, in what follows, we will be mostly interested in the so-called weak asymptotics of the polynomials of the Hermite--Pad\'e $m$-system
(in the spirit of the classical Stahl's theorem for usual Pad\'e polynomials, see \cite{St, ApBuMaSu}),
hence the above discrepancy in the order of contact by a~fixed number $m-1$ is immaterial for us.

\begin{Remark}
Like the Hermite--Pad\'e  polynomials  of type~I and~II, a~polynomial Hermite--Pad\'e $m$-system
can be constructed from any tuple of $m+1$ germs at an arbitrary point $z_0$ on the  Riemann sphere~$\myh\CC$, not only for $z_0=\infty$.
Namely, the $k$th polynomials of the~Hermite--Pad\'e $m$-system  at a~point $z_0\in\CC$
are defined by the relation similar to \eqref{khp_i} (or~\eqref{khp}), where the left-hand side is the same, and the condition $O(z^{-(kn+1)})$
as $z\to\infty$ on the right is replaced the condition $O((z-z_0)^{n(m+1)+1})$ as $z\to z_0$ (which is independent of~$k$).
However, it will be more convenient to consider an Hermite--Pad\'e $m$-system  constructed from analytic germs at infinity.
Nevertheless, all the results discussed below are also true in the general case, with the corresponding change in wording.
\end{Remark}

We will study the weak asymptotics of the above  $k$th polynomials of the Hermite--\allowbreak Pad\'e $m$-system~\eqref{khp_i}
in the case where $f_{j, \infty}$ are germs of functions $f_j$ which are meromorphic  on some compact
$(m+1)$-sheeted Riemann surface $\RS$, and if $\RS$~satisfies some additional condition.
More precisely, let $\RS$ be a~compact Riemann surface, let $\pi\colon \RS\to\myh\CC$ be an  $(m+1)$-sheeted holomorphic branched covering
of the Riemann sphere $\myh\CC$, $m\ge 1$, and let $\Sigma$ be the set of critical values of the projection~$\pi$.
The points of~$\RS$ will be written in boldface,  and their projections will be denoted in lightface
(for example, $\zz\in\RS$, and  $\pi(\zz) = z$).
We denote by $\MM(\RS)$ the space of meromorphic functions on~$\RS$.
Let $f_1, f_2 \dots, f_{m} \in \MM(\RS)$ be such that the functions  $1, f_1, f_2, \dots, f_{m}$
are independent over the field of rational functions~$\mathbb C(z)$.
Let $\circ$ be an arbitrary point on~$\RS$ which is not a~critical point for the projection~$\pi$.
We will assume without loss of generality that $\circ\in\pi^{-1}(\infty)$ and write $\pmb\infty^{(0)}:=\nobreak \circ$.
It should be noted that the case $\infty\in\Sigma$ is not excluded, that is,
it is possible that $\infty$~is a~critical value of~$\pi$, but at the point $\pmb\infty^{(0)}\in\pi^{-1}(\infty)$ the mapping~$\pi$ is nondegenerate.
Let $f_{1, \infty}(z), \dots, f_{m, \infty}(z)$ be the meromorphic germs of the functions $f_1(\zz), \dots, f_{m}(\zz)$ at the point $\pmb\infty^{(0)}$, respectively.
More precisely, $f_{j, \infty}(z):=f_j(\pi^{-1}_0(z))$, where $\pi^{-1}_0$ is the inverse of~$\pi$
in the neighbourhood of the point $\pmb\infty^{(0)}$.
For simplicity of presentation, we assume that the germs $f_{j, \infty}(z)$ are holomorphic at~$\infty$, that is,  they have no poles at this point.
\textit{In what follows, we will consider only polynomials of the Hermite--Pad\'e $m$-system  \eqref{khp_i} constructed from the tuple of such germs
 $[1, f_{1, \infty}, \dots, f_{m, \infty}]$ at the point~$\infty$. Hence, unless otherwise stated, $P_{n,j_1,\dots,j_k}$ will
 denote the corresponding $k$th polynomial of the Hermite--Pad\'e $m$-system constructed from this tuple of germs.}
 It should be emphasized that the number~$m$
  in the definition of an Hermite--Pad\'e $m$-system and in the definition of the covering~$\pi$ is the same.

In the present paper, we will find the limit distribution of the zeros and the
asymptotics of the ratios of the $k$th polynomials
of the Hermite--Pad\'e $m$-system constructed from the above tuple of germs of functions (which are meromorphic on the above Riemann surface $\RS$)
under the following additional condition on~$\RS$.
It can be assumed that the Riemann surface $\RS$ is the standard compactification of the Riemann surface of
an $(m+1)$-valued global analytic function $\pi^{-1}(\cdot)$ defined on the domain $\myh\CC\setminus\Sigma$.
We define the surface $\TRS_{[k]}$ as the standard compactification of the  Riemann surface of all possible
unordered collections of~$k$ distinct germs of the function $\pi^{-1}(\cdot)$ considered at the same
points $z\in\myh\CC\setminus\Sigma$ (for more details, see \S\,\ref{s4}).
We will require that the surface $\TRS_{[k]}$ is connected (this assumption is discussed in \S\,\ref{s6}).

Note that the distribution of the zeros and the asymptotic behaviour of the ratios of the Hermite--Pad\'e polynomials of
type~I that are constructed from the tuple of germs  under consideration (that is, in fact, for
the $m$th polynomials of a~Hermite--Pad\'e $m$-system), have been rigorously justified in~\cite{ChKoPaSu},
and for Hermite--\allowbreak Pad\'e polynomials  of type~II  (that is, for the 1st polynomials of the system),
in~\cite{Nut84} (under some `general position condition').
The  surfaces $\TRS_{[1]}$ and $\TRS_{[m]}$ are isomorphic to~$\RS$ (see \S\,\ref{s6}), and hence, they are always
connected. Thus, in particular, in the present paper, we reprove the result from~\cite{ChKoPaSu} and establish
the result from~\cite{Nut84} in the most general case.
In our study, as in~\cite{ChKoPaSu}, we will use the basic ideas of Nuttall's approach (see \cite{Nut81},~\cite{Nut84}).
Note that our proofs are close in spirit to those of the corresponding results in~\cite{ChKoPaSu},
and the main tools of~\cite{ChKoPaSu} are the methods of potential theory on compact Riemann surfaces.

We will also apply our results obtained here to the problem of reconstruction of values of an algebraic function from its
given germ. In particular, we will show that  if $f_j=f^j$ for some $f\in\MM(\RS)$  and if the surface $\TRS_{[k]}$ is connected,
then the ratio  $P_{n;0,\dots,k-2,k}/P_{n;0,\dots,k-1}$
of the $k$th polynomials of the Hermite--Pad\'e $m$-system
asymptotically (as $n\to\infty$) reconstructs the sum of the values of~$f$ on the first~$k$ sheets
of the Nuttall partition of the Riemann surface $\RS$.

The results of the present paper were partially announced in~\cite{Ko20}.
Unfortunately, in~\cite{Ko20}
 the condition that the surface $\TRS_{[k]}$ should be connected was omitted in the formulation of the corresponding results.
 To date, there is no reason to claim that the results of \S\S\,\ref{s3} and~\ref{s5}
 of the present paper   (in particular, the main Theorems \ref{theorem1} and \ref{theorem2})  are correct without this assumption.

The paper is organized as follows. The main results are formulated in~\S\,\ref{s2}.
In \S\,\ref{s3}, under the condition that all the surfaces $\TRS_{[k]}$, $k=1,\dots, m$, are connected,
we show how using the polynomial Hermite--Pad\'e $m$-system one can
asymptotically reconstruct the values of an arbitrary function $f\in\MM(\RS)$ from its germ on all sheets of the Nuttall partition of~$\RS$
except the last one.
In \S\,\ref{s4}, we give an accurate definition of the surface $\TRS_{[k]}$.
Next,  we give an equivalent definition
of the $k$th polynomials of a~Hermite--Pad\'e $m$-system in terms of the  surface $\TRS_{[k]}$ and
special meromorphic functions on this surface, which are constructed from the original function $f_1,\dots,f_m$.
The main Theorems~\ref{theorem1} and \ref{theorem2} will be proved in~\S\,\ref{s5}. In \S\,\ref{s5.1},
we introduce the necessary definitions, prove the auxiliary results, and fix normalizations.
Theorem~\ref{theorem1} is proved in \S\,\ref{s5.2}, and Theorem~\ref{theorem2}, in~\S\,\ref{s5.3}.
In \S\,\ref{s6}, we discuss the connectedness condition of the  surface $\TRS_{[k]}$
from Theorems~\ref{theorem1} and \ref{theorem2}. In particular, in~\S\,\ref{s6} we give a~sufficient condition for connectedness
of all surfaces $\TRS_{[k]}$ for $k=1,\dots,m$.

\section{Statements of the main results}
\label{s2}

Following Nuttall~\cite{Nut81, Nut84}, we introduce the partition of $\RS$ into sheets (for more details, see \cite{ChKoPaSu}).
Let $u(\zz)$ be a~harmonic function in~$\RS\setminus\pi^{-1}(\infty)$ with the following logarithmic singularities at points
of the set $\pi^{-1}(\infty)$:
\begin{equation}
\label{u_to}
\begin{aligned}
u(\zz)&=-m\log{|z|}+O(1),\quad\zz\to\pmb\infty^{(0)},\\
u(\zz)&=\log{|z|}+O(1),
\quad\zz\to\pi^{-1}(\infty)\setminus\pmb\infty^{(0)}.
\end{aligned}
\end{equation}
The function $u$ always exists  and is defined up to an additive constant (it can be
constructed explicitly using the standard bipolar Green functions; for more details, see formula (23) in~\cite{ChKoPaSu}.

\begin{Remark}
\label{u_branch}
We emphasize that the definition of the  function~$u$ is also correct in the case $\infty\in\Sigma$. (If $\pmb\infty\in\pi^{-1}(\infty)$
is a~critical point of~$\pi$ of order~$N$, then in its neighbourhood~$\OO$ in the local coordinate $\zeta:\OO\to\{\eta:|\eta|<\delta\}$, $\zeta(\pmb\infty)=0$
the second condition from \eqref{u_to} can be written as $u(\zeta^{-1}(\eta))=-N\log{|\eta|}+O(1)$ as $\eta\to 0$.)
\end{Remark}

Let $z\in\CC$ and let $u_0(z),\dots,u_m(z)$  be the values of the function~$u$ at points of the set $\pi^{-1}(z)$
in nondecreasing order  (in the case $z\in\Sigma$, they
are listed the same number of times as the order of the corresponding point of the set $\pi^{-1}(z)$ as a~critical point of~$\pi$):
\begin{equation}
\label{inequ_u}
u_0(z)\le u_1(z)\le\dots\le u_{m-1}(z)\le u_m(z).
\end{equation}
If $u_{j-1}(z)< u_j(z)<u_{j+1}(z)$,
we include in the set $\RS^{(j)}$ (the $j$th sheet of the  surface $\RS$, $j=0,\dots,m$) the
point $\zz^{(j)}\in\pi^{-1}(z)$ such that $u(\zz^{(j)})=u_j(z)$ (for  $j=0$ we consider only the inequality
 $u_0(z)<u_{1}(z)$, and for $j=m$, only the inequality  $u_{m-1}(z)<u_{m}(z)$). Otherwise, points of the set $\pi^{-1}(z)$
  are not included in~$\RS^{(j)}$. For $z=\infty$, we need to replace $u(z)$ in \eqref{inequ_u}  by $u(z)-\log |z|$.
So, the sheets $\RS^{(j)}$ are defined as
\begin{equation}
\label{R^j}
\begin{aligned}
\RS^{(0)}&:=\{\zz\in\RS:0<u_{1}(z)-u(\zz)\};\\
\RS^{(j)}&:=\{\zz\in\RS:u_{j-1}(z)-u(\zz)<0<u_{j+1}(z)-u(\zz)\},\quad j=1,\dots,m-1;\\
\RS^{(m)}&:=\{\zz\in\RS:u_{m-1}(z)-u(\zz)<0\}.
\end{aligned}
\end{equation}
From the definition it follows that  $\RS^{(j)}$ are pairwise disjoint open subsets of $\RS$ (disconnected, in general) and that the projection
$\pi:\RS^{(j)}\to\pi(\RS^{(j)})$ is biholomorphic.
In what follows, the point of the set $\RS^{(j)}$ lying over the point
$z\in \myh\CC$ will be denoted by  $\zz^{(j)}$. The boundary of the sheet $\RS^{(j)}$ will be denoted by $\partial(\RS^{(j)})$.
Since  $u_1(z)-u_0(z)\to+\infty$ as $z\to\infty$, the originally selected point $\pmb\infty^{(0)}$ (at which the
germs of the functions~$f_j$ are considered) always lies in the list $\RS^{(0)}$, this
agrees with our notation for points of sheets. It is clear that  no critical point of the projection~$\pi$ lies in any of the sets $\RS^{(j)}$.

We set
\begin{equation}
\label{F_j}
\begin{aligned}
F_j&:=\{z\in\myh\CC: u_{j-1}(z)=u_j(z)\},\quad j=1,\dots,m,\\
F&:=\cup_{j=1}^mF_j
\end{aligned}
\end{equation}
In Attachment~1 of \cite{ChKoPaSu} it was shown that the sets  $F_j$ and $\partial(\RS^{(j)})$
are  (real) one-\allowbreak dimen\-sio\-nal piecewise analytic sets without isolated points.
The precise definition of a~piecewise analytic set is also given in~\cite{ChKoPaSu}.
Informally speaking, this means that any such set is the closure of the union of a~finite  number of
analytic arcs featuring some regularity at endpoints.
In particular, this implies that the sets  $F_j$ have empty interior, which immediately implies
the equalities $\pi(\partial\RS^{(j)})=F_j\cup F_{j+1}$ for $j=1,\dots,m-1$, and also $\pi(\partial\RS^{(0)})=F_1$ and $\pi(\partial\RS^{(m)})=F_m$.

In the same way as in \cite{ChKoPaSu}, we define on the set $\myh\CC\setminus F$ the matrix~$A$ by
\begin{equation}
\label{A}
A(z):=
\begin{pmatrix}
1 & f_1(\zz^{(0)}) & \ldots & f_m(\zz^{(0)})\\
1 & f_1(\zz^{(1)}) & \ldots & f_m(\zz^{(1)})\\
\hdotsfor{4}\\
1 & f_1(\zz^{(m)}) & \ldots & f_m(\zz^{(m)})
\end{pmatrix}
\end{equation}
(the rows and columns of the matrix~$A$ will be numbered from~$0$ to~$m$).
Clearly, $\det A\in \MM(\myh\CC\setminus F)$ (it  is a~meromorphic function on $\myh\CC\setminus F$).
It is easily seen that $(\det A)^2$ extends to a~meromorphic  function on the whole~$\myh\CC$
(a~crossing of an arc from~$F$ results only in the interchange of some rows  in the matrix~$A$,
that is,  $\det A$ may only change the sign).
Besides that, $\det A \not\equiv 0$, since the functions $1, f_1, f_2, \dots, f_{m}$ are independent over~$\mathbb C(z)$.

Given any $0\le j_1<\dots<j_k\le m$, by $M_{j_1, \dots,j_k}(z)$ we denote
\textit{the minor of the matrix~$A$ corresponding to the columns with numbers $j_1,\dots,j_k$ and
the rows with numbers $0, 1, \dots,k-1$.}
From the definition it follows that $M_{j_1, \dots,j_k}\in \MM(\myh\CC\setminus F)$.
Moreover, if $\TRS_{[k]}$ is connected, then $M_{j_1, \dots,j_k}$ does not vanish identically in any
neighbourhood in $\myh\CC\setminus F$ (see Proposition~\ref{M_ne_0}). Hence,
for any $0\le j_1<\dots<j_k\le m$ and $0\le i_1<\dots<i_k\le m$,
the ratio $M_{j_1, \dots,j_k}(z)/M_{i_1,\dots,i_k}(z)$ is a~meromorphic   function on $\myh\CC\setminus F_k$. Indeed,
a~crossing of arcs from~$F$ results in an~interchange of some rows (with numbers corresponding to the
sheets, whose common boundary is projected to this arc)
in the matrix~$A$.
Moreover, rows with numbers from~0 to~$k-1$ cannot be interchanged with rows with numbers from~$k$ to~$m$
when crossing through $F\setminus F_k$
(since  $u_k\ne u_{k-1}$ on $F\setminus F_k$). Hence, over $F\setminus F_k$,
the sheets with numbers from~0 to~$k-1$ have no common boundary with the sheets with numbers from~$k$ to~$m$.)
So, as $F\setminus F_k$ is crossed, either all the minors $M_{j_1, \dots,j_k}(z)$  remain unchanged or
they all simultaneously change their sign. Hence, on $F\setminus F_k$
the functions $M_{j_1, \dots,j_k}(z)/M_{i_1, \dots,i_k}(z)$ are glued as meromorphic functions.

Let us introduce some notations that will be adhered to in what follows.
By $\xrightarrow{*}$ we will denote the  weak* convergence  (indicating if required the space on which it is
considered). By $\xrightarrow{\mcap}$  we will denote the convergence in the (logarithmic) capacity (indicating if required the set
on which it is considered). Let $d\sigma:= \frac{i}{2\pi}\frac{dz\wedge d\myo z}{(1+|z|^2)^2}$
be the normalized area form of the spherical metric on~$\myh\CC$.
In order to speak about the asymptotic behaviour of the $k$th polynomials of a~Hermite--Pad\'e $m$-system,
we need to fix their normalization. Therefore, along with the polynomials $P_{n;i_1,\dots,i_k}$ (see~\eqref{khp_i}), we will
use the polynomials $P_{n;i_1,\dots,i_k}^*:=c_{n;i_1,\dots,i_k}P_{n;i_1,\dots,i_k}$ ($c_{n;i_1,\dots,i_k}>0$ are constants) for which the functions
$\log|P_{n;i_1,\dots,i_k}^*|$ are spherically normalized:
\begin{equation}
\label{Q_n^*}
\int _{\myh\CC}\log|P_{n;i_1,\dots,i_k}^*|d\sigma=0.
\end{equation}
Note that relation \eqref{khp_i} is not satisfied in general for  $P_{n;i_1,\dots,i_k}^*$.
We will also denote by~$\ddc$ the standard analogue of the Laplace operator on Riemann surfaces;
in general case, this operator transforms currents of degree~0 to currents of degree~2, and on
smooth functions~$\varphi$ it acts in the local coordinate $\zeta = x +iy$ as $\ddc\varphi = (\varphi_{xx}+\varphi_{yy}) dxdy=\Delta\varphi dxdy$.
(For the required properties of the operator~$\ddc$ and results of the theory of potential on compact Riemann surfaces,
see  \cite{ChKoPaSu}, Attachment~2, and~\cite{Ch1}.
We will refer to these facts retaining, if possible, the corresponding notation.)

So, recall  that $P_{n;j_1,\dots,j_k}$ are the  $k$th polynomials of the Hermite--Pad\'e $m$-system
constructed from the tuple of germs $[1, f_{1,\infty},\dots,f_{m,\infty}]$ at the point $\infty$ of the functions $f_j$, $j = 1,\dots, m$, which are
meromorphic  on the Riemann surface~$\RS$. The following results hold.

\begin{theorem}
\label{theorem1}
Suppose that the surface $\TRS_{[k]}$, which is constructed from~$\pi$, is connected. Then

$1)$ There exists a~number $L\in\NN$ such that, for any neighbourhood~$V$ of the compact set  $F_k$, for all sufficiently large~$n$,  $n>N=N(V)$,
there are at most~$L$ zeros of the polynomials $P_{n;i_1,\dots,i_k}$ outside the neighbourhood~$V$.

$2)$ For an arbitrary $p\in[1,\infty)$, as $n\to\infty$,
\begin{equation}
\frac{1}{n}\log|P_{n;i_1,\dots,i_k}^*(z)|\to -\sum_{s=0}^{k-1}u_s(z)
\quad \text{in}\quad L^p(\myh\CC, d\sigma),
\end{equation}
where the function $\sum_{s=0}^{k-1}u_s(z)$ is spherically normalized: $\int _{\myh\CC}\sum_{s=0}^{k-1}u_s(z)d\sigma=0$.

$3)$ As $n\to\infty$, 
\begin{equation}
\frac{1}{n}\ddc \log|P_{n;i_1,\dots,i_k}(z)|\xrightarrow{*} -\ddc \left(\sum_{s=0}^{k-1}u_s(z)\right) \quad\text{in}\quad C(\myh\CC)^*.
\end{equation}
\end{theorem}

\begin{theorem}
\label{theorem2}
Suppose that the surface $\TRS_{[k]}$, which is constructed from~$\pi$, is connected.
Then, for any compact set $K\subset \CC\setminus F_k$, as $n\to\infty$,
\begin{equation}
\label{int_t2_1}
\frac{P_{n;j_1,\dots,j_k}(z)}{P_{n;i_1,\dots,i_k}(z)}
\xrightarrow{\mcap}
\frac{M_{j_1, \dots,j_k}(z)}{M_{i_1, \dots,i_k}(z)},
\quad z\in K.
\end{equation}
Moreover, for an arbitrary $\varepsilon>0$,
\begin{equation}
\mcap\left\{z\in K:\left|\frac{P_{n;j_1,\dots,j_k}(z)}{P_{n;i_1,\dots,i_k}(z)}-
\frac{M_{j_1, \dots,j_k}(z)}{M_{i_1, \dots,i_k}(z)}\right|^{1/n}\cdot
e^{u_{k}(z)-u_{k-1}(z)}\ge 1+\varepsilon\right\}
\to 0.
\end{equation}
\end{theorem}

\section{Reconstruction of the values of a function meromorphic  on $\RS$ from its germ via the polynomial
Hermite--Pad\'e $m$-system}
\label{s3}

In this section, we will consider the following problem. Let $f\in\MM(\RS)$. Assume that we are given some germ
of a~multivalued analytic function $f(\pi^{-1}(z))$ at a~point $z_0\in\myh\CC$ defined in terms of its Taylor series.
(Without loss of generality we assume that $z_0=\infty$.)
The question is, how
to constructively reconstruct the values of~$f$ in `as large a~region as possible' on~$\RS$?
The most obvious way is to use the Weierstrass continuation by reexpanding the Taylor series at points `close' to the boundary of the disk of convergence.
However, this method is not constructive (see~\cite{Hen}).
Another method is to use Pad\'e approximations. By Stahl's theorem (see~\cite{St}), Pad\'e approximations
reconstruct  the values of~$f$ in a~domain $\DD$ on~$\RS$, which projects one-to-one
onto  $\pi(\DD)=\myh\CC\setminus S$, where $S$~is the Stahl compact set, which consists of a~finite number
of analytic arcs (featuring certain regularity at endpoints). So, one can say that the Pad\'e approximations
reconstruct the values of~$f$ on one sheet~$\DD$ of our $(m+1)$-sheeted Riemann surface~$\RS$.
We will show that in the case when all the surfaces $\TRS_{[k]}$, $k=1,\dots,m$, are connected,
the values of~$f$ on all Nuttall sheets, except the `last' $\RS^{(m)}$ and except of the set
$\pi^{-1}(F)$ (see~\eqref{F_j}), are constructively reconstructed with the help of the polynomial Hermite--Pad\'e $m$-system.
It is worth pointing out that the sheet~$\DD$ is not related in any way to the Nuttall sheets of $\RS^{(j)}$.

So, let  $f\in\MM(\RS)$ and let the  functions $1, f, f^2, \dots, f^m$ be independent over the field of rational functions $\mathbb C(z)$.
In what follows, in this section we will consider the polynomial Hermite--Pad\'e $m$-system for $f_j:=f^j$, that is,
this system is constructed from the tuple of germs $[1, f(\pi_0^{-1}(z)), f^2(\pi_0^{-1}(z)),\dots, f^m(\pi_0^{-1}(z))]$, where $\pi_0^{-1}$
is the inverse mapping of~$\pi$ in the neighbourhood of $\pmb\infty^{(0)}$. So, by definition
\begin{equation*}
M_{0,1,\dots,k-1}(z):=\det
\begin{pmatrix}
1 & f(\zz^{(0)}) & \ldots & f^{k-1}(\zz^{(0)})\\
1 & f(\zz^{(1)}) & \ldots & f^{k-1}(\zz^{(1)})\\
\hdotsfor{4}\\
1 & f(\zz^{(k-1)}) & \ldots & f^{k-1}(\zz^{(k-1)})
\end{pmatrix},
\end{equation*}
\begin{equation*}
M_{0,1,\dots,k-2,k}(z):=\det
\begin{pmatrix}
1 & f(\zz^{(0)}) & \ldots & f^{k-2}(\zz^{(0)})& f^{k}(\zz^{(0)})\\
1 & f(\zz^{(1)}) & \ldots & f^{k-2}(\zz^{(1)})& f^{k}(\zz^{(1)})\\
\hdotsfor{5}\\
1 & f(\zz^{(k-1)}) & \ldots & f^{k-2}(\zz^{(k-1)})& f^{k}(\zz^{(k-1)})
\end{pmatrix}.
\end{equation*}
Thus,  $M_{0,1,\dots,k-1}(z)$ is the Vandermonde  determinant, that is,
\begin{equation*}
M_{0,1,\dots,k-1}(z)=\prod_{0\le i<j<k} (f(\zz^{(i)})-f(\zz^{(j)})),
\end{equation*}
and  $M_{0,1,\dots,k-2, k}(z)$ are the matrices
obtained from the Vandermond matrix by increasing the powers of all elements in the last column by~$1$, that is,
\begin{equation*}
M_{0,1,\dots,k-2, k}(z)=\sum_{s=0}^{k-1}f(\zz^{(s)})\cdot\prod_{0\le i<j<k} (f(\zz^{(i)})-f(\zz^{(j)})).
\end{equation*}
Therefore,  $M_{0,\dots,k-2,k}(z)/M_{0, \dots,k-1}(z)
=\sum_{s=0}^{k-1}f(\zz^{(s)})$. The next result now follows from Theorem \ref{theorem2}.

\begin{Corollary}
\label{Cor_rec}
Suppose that the surface $\TRS_{[k]}$, which is constructed from~$\pi$, is connected.
Let $f\in\MM(\RS)$ and let the functions $1, f, f^2, \dots, f^m$ be independent over~$\mathbb C(z)$.
We set $f_j:=f^j$. Then, for any compact set $K\subset \CC\setminus F_k$, as $n\to\infty$
\begin{equation}
\label{int_t2_1}
\frac{P_{n;0,1,\dots,k-2, k}(z)}{P_{n;0,1,\dots,k-1}(z)}
\xrightarrow{\mcap}
\sum_{s=0}^{k-1}f(\zz^{(s)}),
\quad z\in K.
\end{equation}
Moreover, for an arbitrary $\varepsilon>0$,
\begin{equation}
\mcap\left\{z\in K:\left|\frac{P_{n;0,1,\dots,k-2, k}(z)}{P_{n;0,1,\dots,k-1}(z)}-
\sum_{s=0}^{k-1}f(\zz^{(s)})\right|^{1/n}\cdot
e^{u_{k}(z)-u_{k-1}(z)}\ge 1+\varepsilon\right\}
\to 0.
\end{equation}
\end{Corollary}

Let us assume that the projection $\pi$ is such that the  surfaces $\TRS_{[k]}$ are connected for all $k=1,\dots, m$. (In Statement~\ref{St_conn}
it will be shown that this property is satisfied by all~$\pi$ with the condition that all critical points of~$\pi$ are of the first order and
over each point $z\in\myh\CC$ there is at most one critical point of~$\pi$, that is,  the class of such~$\pi$ is quite broad.)
So, evaluating the $k$th polynomials of the Hermite--Pad\'e $m$-system for all $k=1,\dots, m$ and considering
the ratios $P_{n;0,1,\dots,k-2, k}/P_{n;0,1,\dots,k-1}(z)$,
we will in succession asymptotically reconstruct $\sum_{s=0}^{k-1}f(\zz^{(s)})$ outside $\pi^{-1}(F_k)$. Hence (since $F:=\cup_{j=1}^mF_j$),
we will also reconstruct the values of~$f$ on all Nuttall sheets of~$\RS$, except $\RS^{(m)}$, outside
the set $\pi^{-1}(F)$. Since for the evaluation of the $k$th polynomials
of the Hermite--Pad\'e $m$-system
of order~$n$ it suffices to know the first $(m+1)n$ Taylor coefficients of the germs that define the polynomials
(see definition~\eqref{khp_i}), our approximations of $P_{n;0,1,\dots,k-2, k}/P_{n;0,1,\dots,k-1}(z)$ are constructive.
Note that the idea that using suitable polynomials one should reconstruct the sum of the values on the first $k$~sheets,
rather than the values of the function~$f$ on the Nuttall sheets, was expressed for the first time in~\cite{Su}.

\section{The Riemann  surface $\TRS_{[k]}$  and the definition of the $k$th polynomials of the Hermite--Pad\'e $m$-system in terms of this surface}
\label{s4}

First of all, we fix $k\in[1,\dots,m]$. We introduce the compact Riemann surfaces $\TRS_{[k]}$, which is associated
with~$\RS$,  and which is, in general, disconnected; along with $\TRS_{[k]}$ we introduce the branched covering $\tpi:\TRS_{[k]}\to\myh\CC$,
which is constructed from the projection~$\pi$.
Since  $\pi\colon \RS\to\myh\CC$
$(\zz\mapsto z)$ is an $(m+1)$-sheeted holomorphic branched covering of $\myh\CC$ and since $\Sigma$ is the set of
critical values of the projection~$\pi$, we can consider~$\RS$ as the standard compactification of the  Riemann surface $\RS'$
of some $(m+1)$-valued global analytic function (GAF) $w(z)$ defined in the domain $\myh\CC\setminus\Sigma$.
In fact,  $w(z)=\pi^{-1}(z)$. Informally speaking, the surface $\TRS_{[k]}$ is the standard compactification of the Riemann surface $\TRS_{[k]}'$ of all
possible \textit{unordered} collections of  $k$~distinct germs of the function $w(\cdot)$
that are considered at the same points $z\in\myh\CC\setminus\Sigma$.
More precisely, the surface $\TRS_{[k]}'$ consists of the pairs $(z, \{w_1^z,\dots,w^z_{k}\})$, where $z\in\myh\CC\setminus\Sigma$,
and $\{w^z_1,\dots,w^z_{k}\}$ is an \textit{unordered} collection of~$k$ distinct germs of the function  $w(\cdot)$ at the point~$z$.
The structure of the Riemann surface on this set is introduced as in the process of construction
of the Riemann surface of a~GAF. So, by a~neighbourhood
of a~point $(z_0, \{w^{z_0}_1,\dots,w^{z_0}_{k}\})$ we mean the set of points $(z, \{w^{z}_1,\dots,w^{z}_{k}\})$ such that
1)~$z\in B_{z_0}(\delta)$, where $B_{z_0}(\delta)$ is a~disc with centre $z_0$ and a radius~$\delta$
such that the germs $w^{z_0}_1,\dots,w^{z_0}_{k}$ are holomorphic in~$B_{z_0}(\delta)$, and 2)~there exists
a~bijection between elements of the collections $\{w^{z_0}_1,\dots,w^{z_0}_{k}\}$ and $\{w^{z}_1,\dots,w^{z}_{k}\}$ such that the
corresponding  germs are direct analytic continuations of each other. (Since $w(\cdot)$ is an $(m+1)$-sheeted function,
there exists at most one such a~bijection.)
It is clear that  the Riemann surface $\TRS_{[k]}'$ is a~${m+1\choose k}$-sheeted (disconnected, in general)
holomorphic covering of~$\myh\CC\setminus\Sigma$ (with the natural projection $z:(z, \{w^{z}_1,\dots,w^{z}_{k}\})\mapsto z$). Next,
using the standard process of compactification of a~finite-sheeted covering of the Riemann sphere $\myh\CC$
with finite number of punctures, we get from $\TRS_{[k]}'$ a~compact (disconnected, in general) Riemann surfaces $\TRS_{[k]}$ and
a~holomorphic covering $\tpi:\TRS_{[k]}\to\myh\CC$ (branched at points of~$\Sigma$), which extends the original
covering $z:\TRS_{[k]}'\to\myh\CC\setminus\Sigma$.
The points of $\TRS_{[k]}$ will be written in boldface and marked with a~tilde; their projections, as before,
will be denoted in lightface (for example, $\myt\zz\in\TRS_{[k]}$, and  $\myt\pi(\myt\zz) = z$).
We emphasize that, unlike Theorems \ref{theorem1} and \ref{theorem2},
in this section the surface $\TRS_{[k]}$ can be disconnected.

For $z\in\myh\CC\setminus F$ all the inequalities in \eqref{inequ_u} are strict, and hence, for
such~$z$, we have $\pi^{-1}(z)=\{\zz^{(0)},\zz^{(1)},\dots,\zz^{(m)}\}$, where $\zz^{(j)}\in\RS^{(j)}$. So,
over $\myh\CC\setminus F$ the Riemann surface $\TRS_{[k]}$ splits into ${m+1\choose k}$ disjoint sheets
$\TRS_{[k]}^{(j_1 j_2 \dots j_k)}$, $0\le j_1<j_2<\dots<j_k\le m$. The \textit{sheet $\TRS_{[k]}^{(j_1 j_2 \dots j_k)}$} consists of
unordered collections $\{w^{\zz^{(j_1)}}(\cdot), w^{\zz^{(j_2)}}(\cdot),\dots,w^{\zz^{(j_k)}}(\cdot)\}$, where $w^{\zz^{(j_l)}}(\cdot)$
is the germ of $w(\cdot)$ at the point $\zz^{(j_l)}$, which are considered for all $z\in\myh\CC\setminus F$.
It is clear that on all sheets the mapping  $\myt\pi:\TRS_{[k]}^{(j_1 j_2 \dots j_k)}\to \myh\CC\setminus F$ is biholomorphic.
The point of the sheet  $\TRS_{[k]}^{(j_1 j_2 \dots j_k)}$ lying over  $z\in\myh\CC\setminus F$
will be denoted by $\myt\zz^{(j_1 j_2\dots j_k)}$. Note that we have $u_{k-1}(z)<u_k(z)$  for $z\in\myh\CC\setminus F_k$ in~\eqref{inequ_u}.
Therefore,  the unordered tuple   $\{w^{\zz^{(0)}}(\cdot), w^{\zz^{(1)}}(\cdot),\dots,w^{\zz^{(k-1)}}(\cdot)\}$
is defined over $z\in\myh\CC\setminus F_k$. Moreover, extending it along all possible paths in $\myh\CC\setminus F_k$,
we will always get the  same collection. So, the sheet $\TRS_{[k]}^{(01\dots k-1)}$ is defined over  $\myh\CC\setminus F_k$ (not only
over $\myh\CC\setminus F$) and the projection $\myt\pi:\TRS_{[k]}^{(01\dots k-1)}\to \myh\CC\setminus F_k$ is biholomorphic.

Our next aim  is to give a different definition of the $k$th polynomials of a~Hermite--\allowbreak Pad\'e $m$-system in terms of new
functions which are meromorphic on the  surface $\TRS_{[k]}$. Let us construct such functions from the original
functions $f_1,\dots,f_m$.
For each tuples of indices $0\le j_1<j_2<\dots<j_k\le m$ and $0\le l_1<l_2<\dots<l_k\le m$, by
$M_{j_1,\dots,j_k}^{l_1, \dots, l_k}(z)$ we denote the minor of the matrix $A$ (see~\eqref{A})
corresponding to the columns with numbers $j_1,\dots,j_k$ and the rows with numbers $l_1, l_2, \dots,l_k$.
(In particular, $M_{j_1,\dots,j_k}^{0, \dots, k-1}\equiv M_{j_1,\dots,j_k}$.) By  $A_{j_1,\dots,j_k}^{l_1, \dots, l_k}(z)$
we denote the algebraic complement of the minor $M_{j_1,\dots,j_k}^{l_1, \dots, l_k}(z)$ (that is,
$A_{j_1,\dots,j_k}^{l_1, \dots, l_k}(z)$ is the complementary minor to $M_{j_1,\dots,j_k}^{l_1, \dots, l_k}(z)$ taken
with the sign $(-1)^{j_1\dots+j_k + l_1\dots+l_k}$).
Like the matrix~$A$, the minors $M_{j_1,\dots,j_k}^{l_1, \dots, l_k}(z)$ and their algebraic complements
$A_{j_1,\dots,j_k}^{l_1, \dots, l_k}(z)$ are defined for $z\in\myh\CC\setminus F$.
By  $\|M_{j_1,\dots,j_k}^{l_1, \dots, l_k}(z)\|$ and $\|A_{j_1,\dots,j_k}^{l_1, \dots, l_k}(z)\|$, we will denote the matrices
composed of all elements of $M_{j_1,\dots,j_k}^{l_1, \dots, l_k}(z)$ and $A_{j_1,\dots,j_k}^{l_1, \dots, l_k}(z)$, respectively.
It is well known (see, for example, Theorem~2.4.1 in~\cite{Pras}) that since $\det A\not\equiv 0$, then
\begin{equation}
\label{MA_1}
\|M_{j_1,\dots,j_k}^{l_1, \dots, l_k}(z)\|\cdot\left\|\frac{A_{j_1,\dots,j_k}^{l_1, \dots, l_k}(z)}{\det A(z)}\right\|^T = \Id,
\end{equation}
where $\Id$~is the identity matrix of size ${m+1\choose k}\times {m+1\choose k}$.
For $z\in\myh\CC\setminus F$, we introduce the matrix
\begin{equation}
\label{M_w}
M_{w}(z):=\|w^j(\zz^{(l)})\|_{l,j=0}^{m}=
\begin{pmatrix}
1 & w(\zz^{(0)}) & \ldots & w^m(\zz^{(0)})\\
1 & w(\zz^{(1)}) & \ldots & w^m(\zz^{(1)})\\
\hdotsfor{4}\\
1 & w(\zz^{(m)}) & \ldots & w^m(\zz^{(m)})
\end{pmatrix}.
\end{equation}
For each index set $0\le l_1<l_2<\dots<l_k\le m$, we denote by $M_{w}^{l_1, \dots, l_k}(z)$
the minor of the matrix $M_{w}(z)$ corresponding to the columns with numbers $0,1,\dots,k-1$ and the
rows with numbers $l_1, l_2, \dots,l_k$. Note that  $M_{w}^{l_1, \dots, l_k}(z)\not\equiv 0$. Indeed, it is the Vandermonde determinant,
hence $M_w^{l_1, \dots, l_k}(z)=\prod_{1\le i<j\le k}(w(\zz^{(l_j)})-w(\zz^{(l_i)}))$, and $w(\zz^{(s)})\ne w(\zz^{(t)})$ for all $s\ne t$, $z\in\CC\setminus F$, by the definition of $w(z)$ as the algebraic function
defining the surface $\RS$.
Therefore,  for each index set $0\le j_1<j_2<\dots<j_k\le m$, the  following functions are well defined for $\myt\zz\in\TRS_{[k]}\setminus\myt\pi^{-1}(F)$:
\begin{equation}
\label{MM}
M_{j_1,\dots,j_k}(\myt\zz^{(l_1 \dots l_k)}):=\frac{M_{j_1,\dots,j_k}^{l_1, \dots, l_k}(z)}{M_{w}^{l_1, \dots, l_k}(z)},
\end{equation}
\begin{equation}
\label{AA}
A_{j_1,\dots,j_k}(\myt\zz^{(l_1 \dots l_k)}):=\frac{A_{j_1,\dots,j_k}^{l_1, \dots, l_k}(z)M_{w}^{l_1, \dots, l_k}(z)}{\det A}.
\end{equation}

\begin{Proposition}
\label{M_merom}
For each index set $0\le j_1<\dots<j_k\le m$, the functions
$M_{j_1,\dots,j_k}(\myt\zz)$ and $A_{j_1,\dots,j_k}(\myt\zz)$ extend to meromorphic functions on the whole Riemann surface $\TRS_{[k]}$.
\end{Proposition}

\begin{proof}
Let us first prove the conclusion of the proposition for $M_{j_1,\dots,j_k}(\myt\zz)$. By definition~\eqref{MM}
it is clear that  the function $M_{j_1,\dots,j_k}(\myt\zz)$
is meromorphic  on~$\TRS_{[k]}\setminus\myt\pi^{-1}(F)$. Since, as noted before, $F$~is a~one-dimensional piecewise analytic set
without isolated points (that is, in fact,~$F$ is the closure of a~finite number of analytic arcs), the set $\myt\pi^{-1}(F)$
has the same property. Hence it suffices to check that
the functions $M_{j_1,\dots,j_k}(\myt\zz^{(l_1 \dots l_k)})$ and $M_{j_1,\dots,j_k}(\myt\zz^{(i_1 \dots i_k)})$ are glued together in a~single function
when we cross the arcs from the set $\myt\pi^{-1}(F)$ across that we go from the sheet
$\TRS_{[k]}^{(l_1 \dots l_k)}$ to the sheet $\TRS_{[k]}^{(i_1 \dots i_k)}$. Indeed, since across the chosen arc we go from the sheet
$\TRS_{[k]}^{(l_1 \dots l_k)}$ to the sheet $\TRS_{[k]}^{(i_1 \dots i_k)}$, the minor  $M_{j_1,\dots,j_k}^{l_1, \dots, l_k}(z)$
is transformed, up to a~sign, to the minor $M_{j_1,\dots,j_k}^{i_1, \dots, i_k}(z)$ and
the minor $M_{w}^{l_1, \dots, l_k}(z)$ goes, up to a~sign, to $M_{w}^{i_1, \dots, i_k}(z)$
as the projection of this arc is crossed. In addition, the minors $M_{j_1,\dots,j_k}^{i_1, \dots, i_k}(z)$ and $M_{w}^{i_1, \dots, i_k}(z)$ change
or do not change the sign  simultaneously, since the appearance of the sign change depends only
on the structure of the gluing of the sheets of the original surface $\RS$.

Let us prove the proposition for  $A_{j_1,\dots,j_k}(\myt\zz)$.
By~\eqref{MA_1}, we have
\begin{equation}
\label{MA_2}
\|M_{j_1,\dots,j_k}(\myt\zz^{(l_1 \dots l_k)})\|\cdot\|A_{j_1,\dots,j_k}(\myt\zz^{(l_1 \dots l_k)})\|^T=\Id.
\end{equation}
So, for $z\in\CC\setminus F$, the values of the function $A_{j_1,\dots,j_k}(\myt\zz)$
on the sheets of the  surface $\TRS_{[k]}$ are contained in the row numbered by the tuple $j_1<\dots< j_k$, of the
inverse matrix of the matrix whose columns contain the values of the functions $M_{i_1,\dots,i_k}(\myt\zz)$ on the sheets
of the  same surface.
The functions $A_{j_1,\dots,j_k}(\myt\zz)$ defined from the functions  $M_{i_1,\dots,i_k}(\myt\zz)$ in this way
can always be extended to meromorphic functions on $\TRS_{[k]}$.
(It is worth pointing out that for this purpose it is sufficient that the original functions $M_{j_1,\dots,j_k}(\myt\zz)$,
which are meromorphic on $\TRS_{[k]}$, are only linearly independent over~$\CC(z)$ (this is
equivalent to the condition that $\det\|M_{j_1,\dots,j_k}(\myt\zz^{(l_1 \dots l_k)})\|\not\equiv 0$), that is,
it is not necessary that they are minors of the matrix~$A$).
To check this, we need to use the well-known formula, which represents elements of the inverse matrix
as cofactors of the elements of the original matrix divided by its determinant.
Using this formula, one can easily verify that the functions $A_{j_1,\dots,j_k}(\myt\zz)$
are glued into a~single one as the boundaries of the sheets  are crossed.
Note  that the analogous property of a collection of meromorphic functions for
a~Nuttall partition was pointed out in~\cite{Nut81}, see also Proposition~2 in~\cite{ChKoPaSu}.
\end{proof}

Now we can give a~new definition of  the $k$th polynomials of the Hermite--\allowbreak Pad\'e $m$-system.

\begin{theorem}
\label{t3}
There exists a $p\in\NN\cup\{0\}$ {\rm (}not depending on~$n)$  such that the $k$th polynomials of the Hermite--Pad\'e $m$-system $P_{n;i_1,\dots,i_k}$
satisfy the following relations{\rm :}
\begin{equation}
\label{khp_n}
\sum_{0\le i_1<\dots<i_k\le m} P_{n;i_1,\dots,i_k}(z) A_{i_1,\dots,i_k}(\myt\zz^{(l_1 \dots l_k)}) = O\left(\frac{1}{z^{nk+1-p}}\right) \quad \text{as } z\to\infty
\end{equation}
for all  $0\le l_1<l_2<\dots<l_k\le m$, where $l_1\ne 0$.
\end{theorem}

\begin{proof}
Let $P_{n;i_1,\dots,i_k}(z)$ be the $k$th polynomials of the Hermite--Pad\'e $m$-system, that is, the solution of~\eqref{khp_i}.
Define the function
\begin{equation}
\label{tRn}
\myt R_n(\myt\zz):=\sum_{0\le i_1<\dots<i_k\le m} P_{n;i_1,\dots,i_k}(z) A_{i_1,\dots,i_k}(\myt\zz).
\end{equation}
Since $A_{i_1,\dots,i_k}(\myt\zz)\in\MM(\TRS_{[k]})$ (see Proposition~\ref{M_merom}), we have $\myt R_n\in\MM(\TRS_{[k]})$. We need to check that  for $\myt R_n$ the right-hand side of~\eqref{khp_n}
holds true.
We look at the identity
\begin{equation}
\label{tRn_2}
\myt R_n(\myt\zz^{(l_1 \dots l_k)})=\sum_{0\le i_1<\dots<i_k\le m} P_{n;i_1,\dots,i_k}(z) A_{i_1,\dots,i_k}(\myt\zz^{(l_1 \dots l_k)})
\end{equation}
for all $0\le l_1<l_2<\dots<l_k\le m$ as a~system of linear homogeneous equations for $P_{n;i_1,\dots,i_k}(z)$.
The matrices  $\|M_{i_1,\dots,i_k}(\myt\zz^{(l_1 \dots l_k)})\|$ and $\|A_{i_1,\dots,i_k}(\myt\zz^{(l_1 \dots l_k)})\|^T$
are mutually inverse (see \eqref{MA_2}), and hence, solving this system, we have
\begin{equation}
\label{P_M}
P_{n;i_1,\dots,i_k}(z)=\sum_{0\le l_1<\dots<l_k\le m} M_{i_1,\dots,i_k}(\myt\zz^{(l_1\dots l_k)})\myt R_n(\myt\zz^{(l_1\dots l_k)}).
\end{equation}
Substituting these expressions for $P_{n;i_1,\dots,i_k}$ in \eqref{khp_i}, we see that, for any fixed index set
$0<j_1<\dots<j_k\le m$,
\begin{equation}
\label{t3_1}
\begin{aligned}
&\sum_{0\le l_1<\dots<l_k\le m}\Biggl[ M_{j_1,\dots,j_k}(\myt\zz^{(l_1\dots l_k)})+ \\
&\sum_{s=1}^k (-1)^{s}M_{0,j_1,\dots,j_{s-1}, j_{s+1},\dots,j_k}(\myt\zz^{(l_1\dots l_k)})f_{j_s,\infty}(z)
\Biggr]\myt R_n(\myt\zz^{(l_1\dots l_k)})
=O\left(\frac{1}{z^{kn+1}}\right)
\end{aligned}
\end{equation}
as $z\to\infty$. Substituting in \eqref{t3_1} the explicit expression for $M_{i_1,\dots,i_k}(\myt\zz^{(l_1\dots l_k)})$~\eqref{MM} and taking into account that in the neighbourhood of~$\infty$ the germ $f_{j_s,\infty}(z)=f_{j_s}(\zz^{(0)})$, we get
\begin{equation}
\label{t3_2}
\begin{aligned}
&\sum_{0\le l_1<\dots<l_k\le m}\Biggl[M_{j_1,\dots,j_k}^{l_1, \dots, l_k}(z) +\\
&\sum_{s=1}^k (-1)^{s}M_{0,j_1,\dots,j_{s-1}, j_{s+1},\dots,j_k}^{l_1, \dots, l_k}(z)f_{j_s}(\zz^{(0)})
\Biggr] \frac{\myt R_n(\myt\zz^{(l_1\dots l_k)})}{M_{w}^{l_1, \dots, l_k}(z)}
=O\left(\frac{1}{z^{kn+1}}\right)
\end{aligned}
\end{equation}
as $z\to\infty$. Note that  in the square brackets in \eqref{t3_2} we have
the expansion of the minor $M_{0,j_1,\dots,j_k}^{0,l_1, \dots, l_k}(z)$ of the matrix~$A$ along the zero row.
(In particular, for $l_1=0$ we have~$0$ in the square brackets.) Hence \eqref{t3_2} is equivalent to the following
\begin{equation}
\label{t3_3}
\sum_{0< l_1<\dots<l_k\le m}\frac{M_{0,j_1,\dots,j_k}^{0,l_1, \dots, l_k}(z)}{M_{w}^{l_1, \dots, l_k}(z)}
\myt R_n(\myt\zz^{(l_1\dots l_k)})
=O\left(\frac{1}{z^{kn+1}}\right) \quad \text{as } z\to\infty.
\end{equation}

Let $O_\infty\:=\{z\in\myh\CC: |z|>\delta\}$ be  a~neighborhood of infinity such that $O_\infty\cap F_1 = \varnothing$.
(Recall that $F_1$~is the boundary of the sheet  $\RS^{(0)}$ of the  original surface $\RS$.)
Then over~$O_\infty$ sheets $\TRS_{[k]}^{(0 s_2\dots s_k)}$ and $\TRS_{[k]}^{(l_1\dots l_k)}$ for $l_1 > 0$ have no
common boundary points. Hence
$\TRS_{[k]}\cap\myt\pi^{-1}(O_\infty) = \OO_\infty^0\sqcup\OO_\infty^1$, where $\OO_\infty^0:=\myt\pi^{-1}(O_\infty)\cap\overline{\bigcup_{0<l_2<\dots<l_k\le m}\TRS_{[k]}^{(0l_2\dots l_k)}}$ and $\OO_\infty^1:=\myt\pi^{-1}(O_\infty)\cap\overline{\bigcup_{0<l_1<\dots<l_k\le m}\TRS_{[k]}^{(l_1\dots l_k)}}$.
For  $0<j_1<\dots j_k \le m$, we define on $\OO_\infty^1\setminus\myt\pi^{-1}(F)$ the functions  $B_{j_1,\dots,j_k}$ by
\begin{equation}
\label{BB}
B_{j_1,\dots,j_k}(\myt\zz^{(l_1 \dots l_k)}):=\frac{M_{0,j_1,\dots,j_k}^{0,l_1, \dots, l_k}(z)}{M_{w}^{l_1, \dots, l_k}(z)}.
\end{equation}
Since 1) for fixed  $0<l_1<\dots l_k\le m$, the functions $B_{j_1,\dots,j_k}(\myt\zz^{(l_1 \dots l_k)})$ are identically equal
to the expression in the square brackets in \eqref{t3_1}, 2) by Proposition~\ref{M_merom},
all functions $M_{i_1,\dots,i_k}(\myt\zz)$ are meromorphic on $\TRS_{[k]}$ (and, in particular, on $\OO_\infty^1$),
and 3) the germs $f_{i,\infty}$ are holomorphic on~$O_\infty$,  the functions  $B_{j_1,\dots,j_k}(\myt\zz)$
extend to meromorphic  functions on~$\OO_\infty^1$.

Let us now look at expression \eqref{t3_3} for all $0< j_1<\dots<j_k\le m$  as a~system of linear homogeneous equations for
$\myt R_n(\myt\zz^{(l_1\dots l_k)})$ for  $0< l_1<\dots<l_k\le m$, assuming that $\myt\zz\in\OO_\infty^1\setminus\myt\pi^{-1}(F)$.
First of all, we evaluate the determinant of the matrix of this system
$\|B_{j_1,\dots,j_k}(\myt\zz^{(l_1 \dots l_k)})\|$. By the generalized Sylvester identity (see, for example,
\S~2.7 in~\cite{Pras}), which expresses the determinant of the matrix composed of all minors
of a~given order and containing a~fixed corner minor (in our case, an element) in terms of the determinant of the original matrix, we have
$\det\|M_{0,j_1,\dots,j_k}^{0,l_1, \dots, l_k}(z)\|=(\det A)^{{m-1\choose k-1}}$. Therefore,
\begin{equation}
\det \|B_{j_1,\dots,j_k}(\myt\zz^{(l_1 \dots l_k)})\| = \det\left\|\frac{M_{0,j_1,\dots,j_k}^{0,l_1, \dots, l_k}(z)}{M_{w}^{l_1, \dots, l_k}(z)}\right\|=
\frac{(\det A)^{{m-1\choose k-1}}}{\prod\limits_{0<l_1<\dots<l_k\le m}M_{w}^{l_1, \dots, l_k}(z)}.
\end{equation}
Since, as pointed out above, $\det A\not\equiv 0$ (because the original functions $1, f_1,\dots,f_m$ are linearly independent over~$\CC(z)$) and $M_{w}^{l_1, \dots, l_k}(z)\not\equiv 0$
for all $0\le l_1<\dots<l_k\le m$ (see the definition of the
functions $M_{j_1,\dots,j_k}(\myt\zz^{(l_1\dots l_k)})$~\eqref{MM}),
we also have $\det \|B_{j_1,\dots,j_k}(\myt\zz^{(l_1 \dots l_k)})\|\not\equiv 0$. Hence the
matrix $\|B_{j_1,\dots,j_k}(\myt\zz^{(l_1 \dots l_k)})\|$ is invertible. Moreover, since
the columns of the matrix $\|B_{j_1,\dots,j_k}(\myt\zz^{(l_1 \dots l_k)})\|$ contain the
values on the sheets of the surface $\OO_\infty^1$ of functions that are meromorphic  on this surface, we obtain that
the rows of the inverse matrix $\|B_{j_1,\dots,j_k}(\myt\zz^{(l_1 \dots l_k)})\|^{-1}$ contain
values on the sheets of the  surface $\OO_\infty^1$ of functions that are also meromorphic  on this surface.
(This property, which can be verified directly, is, in fact, established in the proof of Proposition~\ref{M_merom}
for the functions $A_{j_1,\dots,j_k}(\myt\zz)$: here one only needs to replace the surface $\TRS_{[k]}$ by~$\OO_\infty^1$.)
So, the singularities of the matrix $\|B_{j_1,\dots,j_k}(\myt\zz^{(l_1 \dots l_k)})\|^{-1}$ over~$\infty$ are `finite'.
More precisely, let $O\subset O_\infty$ be a~neighborhood of~$\infty$ such that on the set $\myt\pi^{-1}(O)$ the poles
of the meromorphic functions, contained in the rows of the matrix $\|B_{j_1,\dots,j_k}(\myt\zz^{(l_1 \dots l_k)})\|^{-1}$,
may lie only at points of the set $\myt\pi^{-1}(\infty)$. Then  there exists a~$p\in\NN\cup \{0\}$
such that the matrix  $z^p\|B_{j_1,\dots,j_k}(\myt\zz^{(l_1 \dots l_k)})\|^{-1}$ (that is, the matrix obtained from
 $\|B_{j_1,\dots,j_k}(\myt\zz^{(l_1 \dots l_k)})\|^{-1}$ by multiplication of each of its element by~$z^p$) is bounded in~$O$.
Solving system~\eqref{t3_3} for $\myt R_n(\myt\zz^{(l_1\dots l_k)})$, we arrive at the conclusion of the theorem.
\end{proof}

\begin{Remark}
In the case $\infty\notin\Sigma$ and if all the functions $A_{i_1,\dots,i_k}(\myt\zz)$
have no poles at points of the set $\myt\pi^{-1}(\infty)$, there always exist polynomials $\myt P_{n;i_1,\dots,i_k}(z)$
of degree at most $(m+1-k)n$ satisfying~\eqref{khp_n} for $p=0$. Indeed, it is easily seen that in this case
conditions \eqref{khp_n}, as well as conditions \eqref{khp_i}, form a~system of $(n(m+1)+1){m\choose k}$ linear homogeneous equations for $(n(m+1-k)+1){m+1\choose k}=n(m+1){m\choose k}+{m+1\choose k}$ unknown coefficients of the polynomials $\myt P_{n;i_1,\dots,i_k}(z)$.
Moreover, in the case $\infty\notin\Sigma$ all the germs $A_{i_1,\dots,i_k}(\myt\zz^{(l_1\dots l_k)})$ for
$0\le i_1<\dots<i_k\le m$ and $0< l_1<\dots<l_k\le m$ can be replaced by a~tuple of arbitrary holomorphic germs, and still
polynomials  $\myt P_{n;i_1,\dots,i_k}(z)$ of degree $\le (m+1-k)n$ satisfying~\eqref{khp_n} for $p=0$
will exist. (However, they will have nothing to do with our original problem).

However, in the case   $\infty\in\Sigma$ the situation is completely different.
In this case, we cannot \textit{a~priori}  (without recourse to Theorem \ref{t3}) assert that   the
polynomials $\myt P_{n;i_1,\dots,i_k}(z)$  exist even if the right-hand side of \eqref{khp_n}
is replaced by $O(z^{-1})$ and the functions  $A_{i_1,\dots,i_k}(\myt\zz)$ are holomorphic at  points of the set  $\myt\pi^{-1}(\infty)$.
This is so, because the functions $A_{i_1,\dots,i_k}(\myt\zz)$, as considered as functions of  $z\in\myh\CC$,
may have branch points at points of the set $\myt\pi^{-1}(\infty)$.
\end{Remark}

The following  corollary shows that, for some `general position condition', system \eqref{khp_n} for $p=0$ gives
a~new definition of  the $k$th polynomials of the Hermite--Pad\'e $m$-sys\-tem, which is
equivalent to the original one. (All the notation used in this definition is inherited from the proof  of Theorem~\ref{t3}.)

\begin{Corollary}
Let $\infty\notin\Sigma$ and let  the functions $B_{j_1,\dots,j_k}(\myt\zz)$ and the functions from the rows of the matrix $\|B_{j_1,\dots,j_k}(\myt\zz^{(l_1 \dots l_k)})\|^{-1}$
have no poles on the set $\myt\pi^{-1}(\infty)$. Then  the $k$th polynomials
of the Hermite--Pad\'e $m$-system {\rm (}as defined in \eqref{khp_i}{\rm )} and the polynomials
of degree  at most $(m+1-k)n$ satisfying~\eqref{khp_n} for $p=0$ coincide.
\end{Corollary}

\begin{proof}
Let $P_{n;i_1,\dots,i_k}(z)$ be the $k$th polynomials of the Hermite--Pad\'e $m$-system, that is, the solution of~\eqref{khp_i}.
Since the functions contained in the rows of the matrix $\|B_{j_1,\dots,j_k}(\myt\zz^{(l_1 \dots l_k)})\|^{-1}$
have no poles on the set $\myt\pi^{-1}(\infty)$, we can put $p=0$ in the proof of Theorem~\ref{t3}.

Conversely, let $\myt P_{n;i_1,\dots,i_k}(z)$ be polynomials  of degree at most  $(m+1-k)n$ satisfying~\eqref{khp_n} for $p=0$.
Consider on $\TRS_{[k]}$ the meromorphic  function
\begin{equation}
\label{tRn'}
\myt R'_n(\myt\zz):=\sum_{0\le i_1<\dots<i_k\le m}\myt P_{n;i_1,\dots,i_k}(z) A_{i_1,\dots,i_k}(\myt\zz).
\end{equation}
Working in the same way as in the proof of Theorem \ref{t3}, we express $\myt P_{n;i_1,\dots,i_k}(z)$ from \eqref{tRn'}:
\begin{equation}
\myt P_{n;i_1,\dots,i_k}(z)=\sum_{0\le l_1<\dots<l_k\le m} M_{i_1,\dots,i_k}(\myt\zz^{(l_1\dots l_k)})\myt R'_n(\myt\zz^{(l_1\dots l_k)}).
\end{equation}
Substituting these expressions for  $\myt P_{n;i_1,\dots,i_k}(z)$ in the left-hand side of \eqref{khp_i},
carrying  out the same transformations as in the proof
of Theorem~\ref{t3} in \eqref{t3_1}--\eqref{t3_3}, and using \eqref{BB}, we find that
\begin{equation}
\label{P_P}
\begin{aligned}
&\myt P_{n;j_1\dots,j_k}(z) +
\sum_{s=1}^k (-1)^{s}\myt P_{n;0,j_1,\dots,j_{s-1}, j_{s+1},\dots,j_k}(z)f_{j_s,\infty}(z) \\
&=\sum_{0< l_1<\dots<l_k\le m}B_{j_1,\dots,j_k}(\myt\zz^{(l_1 \dots l_k)})\myt R'_n(\myt\zz^{(l_1\dots l_k)}).
\end{aligned}
\end{equation}
Note that by the assumption the functions  $B_{j_1,\dots,j_k}(\myt\zz)$ have no poles at the points of the set $\myt\pi^{-1}(\infty)$.
Besides, condition~\eqref{khp_n} for $p=0$ is equivalent to the relation
$\myt R'_n(\myt\zz^{(l_1 \dots l_m)})=O(z^{-(nk+1)})$ as $z\to \infty$ for all $0< l_1<\dots<l_k\le m$. Hence from \eqref{P_P} we have
\begin{equation}
\myt P_{n;j_1\dots,j_k}(z) +
\sum_{s=1}^k (-1)^{s}\myt P_{n;0,j_1,\dots,j_{s-1}, j_{s+1},\dots,j_k}(z)f_{j_s,\infty}(z)=O\left(\frac{1}{z^{nk+1}}\right)
\end{equation}
as $z\to\infty$, which is condition~\eqref{khp_i}.
\end{proof}

\begin{Proposition}
\label{M_ne_0}
If the surface $\TRS_{[k]}$ is connected, then, for any index set $0\le j_1<\dots<j_k\le m$, the meromorphic function
$M_{j_1,\dots,j_k}(\myt\zz)$ on this surface is not identically $0$ and the function $M_{j_1,\dots,j_k}(z)$
does not identically vanish in any domain in $\myh\CC\setminus F$.
\end{Proposition}

\begin{proof}
Assume that, for some index set $0\le j_1<\dots<j_k\le m$, we have $M_{j_1,\dots,j_k}(\myt\zz)\equiv 0$
in some domain on~ $\TRS_{[k]}$. Since $\TRS_{[k]}$ is connected, this means that $M_{j_1,\dots,j_k}(\myt\zz)\equiv 0$ on the whole $\TRS_{[k]}$. Therefore,  $\det\|M_{j_1,\dots,j_k}(\myt\zz^{(l_1 \dots l_k)})\|\equiv 0$ for $z\in\myh\CC\setminus F$.
On the other hand, since $\|M_{j_1,\dots,j_k}^{l_1, \dots, l_k}(z)\|$ is the matrix composed of all minors of~$A$ of size  $k\times k$, we have
\begin{equation}
\det\|M_{j_1,\dots,j_k}(\myt\zz^{(l_1 \dots l_k)})\|=
\det\left\|\frac{M_{j_1,\dots,j_k}^{l_1, \dots, l_k}(z)}{M_{w}^{l_1, \dots, l_k}(z)}\right\|=
\frac{(\det A(z))^{m\choose k-1}}{\prod\limits_{0\le l_1<\dots<l_k\le m}M_{w}^{l_1, \dots, l_k}(z)}.
\end{equation}
Therefore,  $\det A\equiv 0$ on $\myh\CC\setminus F$. This contradicts the linear independence over~$\CC(z)$
of the original functions $1, f_1,\dots, f_m$.

The second assertion of the proposition is clear from the fact that, for $z\in\myh\CC\setminus F$,
\begin{equation}
M_{j_1,\dots,j_k}(z)=M_{j_1,\dots,j_k}(\myt\zz^{(01 \dots k-1)}){M_{w}^{0,1, \dots, k-1}(z)}.
\end{equation}
\end{proof}

\section{Proofs of Theorems~\ref{theorem1} and \ref{theorem2}}
\label{s5}

As pointed out above, our proofs of Theorems~\ref{theorem1} and \ref{theorem2} are close to the proofs of the analogous theorems
for the Hermite--Pad\'e polynomials  of type~I from~\cite{ChKoPaSu}.
Roughly speaking, we will replace the function $u(\zz)$, which defines the Nuttall partition  of the  surface $\RS$,
by the function $-\myt u(\myt\zz)$ on~$\TRS_{[k]}$, which we will now construct from~$u$,
and the remainder function  for Hermite--Pad\'e polynomials of type~I by $\myt R_n(\myt\zz)$, which is defined by~\eqref{tRn}.

\subsection{Auxiliary results}
\label{s5.1}

So, for $\myt\zz\in\TRS_{[k]}\setminus\myt\pi^{-1}(F)$, we define the function
\begin{equation}
\label{uu}
\myt u(\myt\zz^{(l_1 \dots l_k)}):=u_{l_1}(z)+\dots +u_{l_k}(z),
\end{equation}
where $u_j(z)$ are introduced in \eqref{inequ_u}.

\begin{Proposition}
The function $\myt u(\myt\zz)$ extends to a~harmonic function on  $\TRS_{[k]}\setminus\myt\pi^{-1}(\infty)$
with the following logarithmic singularities at the points $\myt\pi^{-1}(\infty)$:
\begin{equation}
\label{uu_to}
\begin{aligned}
\myt u(\zz^{(0l_2\dots l_k)})&=-(m+1-k)\log{|z|}+O(1),\quad z\to\infty,\\
\myt u(\zz^{l_1 l_2\dots l_k})&=k\log{|z|}+O(1), \quad z\to\infty, \quad l_1\ne 0.
\end{aligned}
\end{equation}
\end{Proposition}

\begin{proof}
As pointed out above, $F$ is a~one-dimensional piecewise analytic set without isolated points (that is, in fact, $F$
is the closure of a~finite number of analytic arcs), and hence the set $\myt\pi^{-1}(F)$ has the same properties.
So, it suffices to show that the function $\myt u$ is glued into a~single harmonic function on the arcs from $\myt\pi^{-1}(F)$.
Assume that when crossing such an  (open) arc $\gamma\in\myt\pi^{-1}(F)$ we go from the sheet  $\TRS_{[k]}^{(l_1 \dots l_k)}$ to the sheet
 $\TRS_{[k]}^{(i_1 \dots i_k)}$. Consider a~small neighborhood~$O$ of the arc  $\myt\pi(\gamma)$ such that  $O\cap F = \myt\pi(\gamma)$. Over~$O$ the Riemann
 surface $\RS$ splits into the union of   $m+1$ disconnected neighbourhoods $\OO_s$, $s=0,\dots,m$. We set $u_{(s)}(z):=u|_{\OO_s}(\zz)$.
 Then, from each side of $\myt\pi(\gamma)$, the function~$u_j$ coincides with some of the functions $u_{(s)}$. Besides, if the
 set $\OO_s\cap\pi^{-1}(\myt\pi(\gamma))$ locally separates some sheets $\RS^{(l)}$ and $\RS^{(i)}$ (the case $l=i$
 is possible), then the functions $u_l$ and $u_i$ are glued on  $\myt\pi(\gamma)$ into a~single harmonic function $u_{(s)}$.
 Hence, by construction of the  surfaces $\TRS_{[k]}$, when we go through the arc  $\myt\pi(\gamma)$, the functions  $u_{l_1}(z)+\dots +u_{l_k}(z)$ and $u_{i_1}(z)+\dots +u_{i_k}(z)$ are glued into a~single harmonic function. This means that on~$\gamma$ the function~$\myt u$
 is glued into a~harmonic function.

The behaviour of $\myt u$ at points of the set $\myt\pi^{-1}(\infty)$ \eqref{uu_to}
directly follows from the asymptotic behaviour of~$u$ \eqref{u_to} and the definition of the  surface $\TRS_{[k]}$.
\end{proof}

Note that  $\myt u$, as well as the function $u$~\eqref{u_to} that defines it, is defined up to an additive constant, which will be chosen later.

As noted above, the function $\myt R_n$, defined by equality~\eqref{tRn}, is a~meromorphic function on $\TRS_{[k]}$.
By the assumption, the surface $\TRS_{[k]}$ is connected, and hence either $\myt R_n\equiv 0$
on the whole $\TRS_{[k]}$, or $\myt R_n$ vanishes at a~finite number of  points. Let us show that  the first case is
impossible. Assume the converse, i. e. $\myt R_n\equiv 0$. Then, since  $\det\|A_{j_1,\dots,j_k}(\myt\zz^{(l_1 \dots l_k)})\|\not\equiv0$ (see \eqref{MA_2}),
from \eqref{tRn_2} it follows that all $P_{n;i_1,\dots,i_k}\equiv 0$. This contradicts their definition~\eqref{khp_i}. So,
$\myt R_n\in\MM(\TRS_{[k]})$ and $\myt R_n\not\equiv 0$. Let us write down its divisor $(\myt R_n)$.

We denote by $p_1$ the maximum of the orders of the poles of all functions $A_{j_1,\dots,j_k}$~\eqref{AA} at all points of the set  $\myt\pi^{-1}(\infty)$.
We have $\pmb\infty^{(0)}\not\in\Sigma$, and hence, moving on the  surface $\TRS_{[k]}$ near  $\myt\pi^{-1}(\infty)$,
we can not go from the sheets $\TRS_{[k]}^{(0 s_2\dots s_k)}$ to the sheets $\TRS_{[k]}^{(l_1\dots l_k)}$, where $l_1 > 0$.  Assume
first that  $\infty\not\in\Sigma$. By Theorem~\ref{t3}, in the set $\myt\pi^{-1}(\infty)$
there are precisely  $m\choose k$ points, near which $\myt R_n$ behaves as  $O(z^{-(nk+1-p)})$. (Such points
lie either on the sheets $\TRS_{[k]}^{(l_1\dots l_k)}$, where $l_1 > 0$, themselves or on their
boundaries in the case $\infty\in F$.) Therefore,
at these points  $\myt R_n$ has a~zero of order $\ge nk+1-p$.
By definition, $\deg P_{n;j_1,\dots,j_m}\le (m+1-k)n$, and hence at other ${m\choose k-1}$ points
of the set $\myt\pi^{-1}(\infty)$ the function $\myt R_n$ has a~pole of order  $\le (m+1-k)n+p_1$.
Assume now that the case $\infty\in\Sigma$ is allowed. We denote by $\myt{\pmb\infty}_1, \dots, \myt{\pmb\infty}_{M_1}$
the points of the set $\myt\pi^{-1}(\infty)$ that lie either on sheets $\TRS_{[k]}^{(l_1\dots l_k)}$, where $l_1 > 0$,
or on their boundaries. We also denote by $\myt{\pmb\infty}_{M_1+1}, \dots, \myt{\pmb\infty}_{M_2}$ other points  from  $\myt\pi^{-1}(\infty)$. Let $d_j$ be the order of $\myt{\pmb\infty}_j$ as a~critical point
of~$\myt\pi$ ($d_j=1$ if $\myt{\pmb\infty}_j$ is not a~critical point). By
Theorem~\ref{t3},  the function  $\myt R_n$
has a~zero of order $\ge d_j(nk+1-p)$ at~$\myt{\pmb\infty}_j$ for $j=1,\dots M_1$, and  
 $\sum_{j=1}^{M_1}d_j = {m\choose k}$. We have  $\deg P_{n;j_1,\dots,j_m}\le (m+1-k)n$,
and hence the function  $\myt R_n$ has  a~pole of order  $\le d_j(m+1-k)n+p_1$  at~$\myt{\pmb\infty}_j$ for  $j=M_1+1,\dots M_2$, and  
$\sum_{j=M_1+1}^{M_2}d_j = {m\choose k-1}$.

Let $\{\myt{\pmb\alpha}_j(n)\}_{j=1}^{S_1(n)}$ be the zeros of $\myt R_n$ over $\CC$ (taken with
multiplicities) and let $\{\myt{\pmb\beta}_j(n)\}_{j=1}^{S_2(n)}$ be its poles (taken with
multiplicities). Then
\begin{equation}
\label{div_tRn_0}
\begin{aligned}
&(\myt R_n) =
\sum_{j=1}^{M_1}(d_j(nk+1-p)+r_j(n))\myt{\pmb\infty}_j  \\
&-\sum_{j=M_1+1}^{M_2}(d_j(m+1-k)n+p_1-r_j(n))\myt{\pmb\infty}_j
+\sum_{j=1}^{S_1(n)}\myt{\pmb\alpha}_j(n)-\sum_{j=1}^{S_2(n)}\myt{\pmb\beta}_j(n),
\end{aligned}
\end{equation}
where $r_j(n)\in\NN\cup\{0\}$.
We rewrite~\eqref{div_tRn_0} in the form
\begin{equation}
\label{div_tRn_tTn}
(\myt R_n)=n\left(k\sum_{j=1}^{M_1}d_j\myt{\pmb\infty}_j - (m+1-k)\sum_{j=M_1+1}^{M_2}d_j\myt{\pmb\infty}_j\right) +(\myt T_n),
\end{equation}
where
\begin{equation}
\label{div_tTn}
\begin{aligned}
(\myt T_n):=&(1-p)\sum_{j=1}^{M_1}d_j\myt{\pmb\infty}_j -p_1\sum_{j=M_1+1}^{M_2}\myt{\pmb\infty}_j + \sum_{j=1}^{M_2}r_j(n)\myt{\pmb\infty}_j \\
&+\sum_{j=1}^{S_1(n)}\myt{\pmb\alpha}_j(n)-\sum_{j=1}^{S_2(n)}\myt{\pmb\beta}_j(n).
\end{aligned}
\end{equation}
Since  $\sum_{j=1}^{M_1}d_j = {m\choose k}$, and $\sum_{j=M_1+1}^{M_2}d_j = {m\choose k-1}$, the  identity $\deg(\myt R_n)=0$
is equivalent to $\deg(\myt T_n)=0$. Let us look at the negative points of the divisor of $(\myt T_n)$. First of all, these are the
points $\myt{\pmb\beta}_j(n)$.
Over~$\CC$ the function $\myt R_n$ may have poles only at points that are a pole of some of the functions $A_{j_1,\dots,j_k}$
and the order of this pole is not greater than the maximum of the order of the poles of all
functions $A_{j_1,\dots,j_k}$ at this point. Hence the set $\{\myt{\pmb\beta}_j(n)\}_{j=1}^{S_2(n)}$ is a~subset of the set of the poles
of all functions $A_{j_1,\dots,j_k}$ over~$\CC$ (with multiplicities).
All other negative points lie over~$\infty$ and do not depend on~$n$.
Therefore, adding, if necessary, the missing points to the divisor of~$(\myt T_n)$ with different signs, we
can assume that the negative points of $(\myt T_n)$ are precisely the poles of all functions $A_{j_1,\dots,j_k}$ (with multiplicities)
and the points $\myt{\pmb\infty}_j$ with the multiplicities given by \eqref{div_tTn}. So, now the negative points of~$(\myt T_n)$ do not depend on~$n$. We denote these points by $\myt\bb_j$, $j=1,\dots S$ (the sign is not taken into account), where $S$~is their total
number (taken with multiplicities). We have $\deg(\myt T_n)=0$, and hence now  $(\myt T_n)$ has precisely~ $S$ unknown zeros,
which we denote by $\myt\aa_j(n)$, $j=1,\dots, S$. As a~result, we have
\begin{equation}
(\myt T_n) =
\sum_{j=1}^{S}\myt\aa_j(n)-\sum_{j=1}^{S}\myt\bb_j.
\end{equation}
Substituting this expression in \eqref{div_tRn_tTn}, we get
\begin{equation}
\label{div_tRn}
(\myt R_n) =
n\left(k\sum_{j=1}^{M_1}d_j\myt{\pmb\infty}_j - (m+1-k)\sum_{j=M_1+1}^{M_2}d_j\myt{\pmb\infty}_j\right)+
\sum_{j=1}^{S}\myt\aa_j(n)-\sum_{j=1}^{S}\myt\bb_j.
\end{equation}

For any two distinct points $\myt\qq, \myt\pp\in\TRS_{[k]}$, by $g(\myt\qq, \myt\pp; \zz)$ we denote the standard bipolar Green function, that is,
a~harmonic function on $\TRS_{[k]}\setminus\{\myt\qq, \myt\pp\}$ with logarithmic singularities at~$\myt\qq$ and $\myt\pp$ that in the local coordinates~$\zeta$ have the following form
\begin{equation}
\label{g_as}
\begin{aligned}
g(\myt\qq,\myt\pp;\myt\zz)&=\log{|\zeta(\myt\zz)-\zeta(\myt\qq)|}+O(1),\quad \zz\to\myt\qq, \\
g(\myt\qq,\myt\pp;\myt\zz)&=-\log|\zeta(\myt\zz)-\zeta(\myt\pp)|+O(1),\quad \myt\zz\to\myt\pp.
\end{aligned}
\end{equation}
The existence of bipolar Green functions on any compact Riemann surface (which is well known), and
all their properties we need are proved in~\cite{ChKoPaSu} (see also~\cite{Ch1}). The functions $g(\myt\qq,\myt\pp;\myt\zz)$ are still defined up to an additive constant. We will choose their normalization later.

From \eqref{div_tRn},~\eqref{uu_to} and since $\TRS_{[k]}$ is connected, it clearly follows that
\begin{equation}
\label{log_tRn}
\log |\myt R_n(\myt\zz)|=-n\myt u(\myt\zz) + \sum_{j=1}^S g(\myt\aa_j(n), \myt\bb_j;\myt\zz) + c_n,
\end{equation}
where $c_n$ is a real constant.
Setting
\begin{equation}
\label{psi_n}
\psi_n(\myt\zz):=\exp\biggl\{\sum_{j=1}^S g(\myt\aa_j(n), \myt\bb_j;\myt\zz)
\biggr\},
\end{equation}
we have
\begin{equation}
\label{|tRn|}
|\myt R_n(\myt\zz)|=C_n e^{-n\myt u(\myt\zz)}\psi_n(\myt\zz),
\end{equation}
where $C_n=e^{c_n}>0$ is a~constant.

Let us now fix normalizations of the functions $g, \myt u, \myt R_n$ and the $k$th polynomials of the Hermite--Pad\'e $m$-system $P_{n;i_1,\dots,i_k}$.
The functions $g$ and $\myt u$ are defined up to an additive constant, and hence the choice of their
normalizations is equivalent to the choice of a~multiplicative constant~$C_n$ in~\eqref{|tRn|}. Since~$\myt R_n$
is expressed in terms of $P_{n;i_1,\dots,i_k}$ by~\eqref{tRn}, it follows that with the help of multiplication of all $P_{n;i_1,\dots,i_k}$ by the same
constant we can get any suitable normalization of~$\myt R_n$. In what follows, we will assume that all the functions~$g$ and~$\myt u$
are spherically normalized on the sheet $\TRS_{[k]}^{(01\dots k-1)}$, that is,
\begin{equation}
\label{norm_g}
\int _{\myh\CC\setminus F_k}g(\myt\qq, \myt\pp; \myt\zz^{(01\dots k-1)})d\sigma(z)=0, \quad
\int _{\myh\CC\setminus F_k}\myt u(\myt\zz^{(01\dots k-1)})d\sigma(z)=0,
\end{equation}
where $d\sigma:= \frac{i}{2\pi}\frac{dz\wedge d\myo z}{(1+|z|^2)^2}$ is the normalized area form of the spherical metric on~$\myh\CC$.
To fix the normalization of~$|\myt R_n|$, we set $c_n=0$ in \eqref{log_tRn}, which is equivalent to setting $C_n=1$ in \eqref{|tRn|}, that is,
\begin{equation}
\label{|tRn|_2}
|\myt R_n(\myt\zz)|=e^{-n\myt u(\myt\zz)}\psi_n(\myt\zz).
\end{equation}
So, $\log|\myt R_n|$ is also spherically normalized on the sheet $\TRS_{[k]}^{(01\dots k-1)}$.
Thus, we have fixed the normalization of the $k$th polynomials of the Hermite--Pad\'e $m$-system.
\textit{In what follows, $P_{n;i_1,\dots,i_k}$ will denote exactly those $k$th polynomials
of the Hermite--\allowbreak Pad\'e $m$-system for that the function $\log|\myt R_n|$ is spherically normalized on the sheet $\TRS_{[k]}^{(01\dots k-1)}$.}

Note that, in general, the functions $\log|P_{n;i_1,\dots,i_k}|$ are not spherically normalized. Therefore  along with $P_{n;i_1,\dots,i_k}$
we will consider the polynomials $P_{n;i_1,\dots,i_k}^*=c_{n;i_1,\dots,i_k}P_{n;i_1,\dots,i_k}$ such that
\begin{equation}
\label{P^*}
\int _{\myh\CC}\log|P_{n;i_1,\dots,i_k}^*|d\sigma=0,
\end{equation}
where $c_{n; i_1,\dots,i_k}>0$ are constants.
We emphasize that, in general, $P_{n;i_1,\dots,i_k}^*$ do not satisfy~\eqref{khp_i}.

We fix some conformal metric $\rho$ on $\TRS_{[k]}$. We denote by $\dist_\rho(\cdot,\cdot)$ the distance with respect to this metric, and by~$\sigma_\rho$ the corresponding area form.
Since the spaces~$L^p$ corresponding to area forms of any two smooth positive Riemannian metrics on a~compact Riemann surface
coincide, we denote by $L^p(\TRS_{[k]})$ the space~$L^p$ corresponding to the area form~$\sigma_\rho$.
For  the Riemann sphere~$\myh\CC$ and its subsets we will consider the spaces $L^p$ that correspond to the normalized area form $d\sigma$
of the spherical metric~$\rhos$, which will be not specified explicitly.
By $\dist(\cdot,\cdot)$, we will denote the distance on~$\myh\CC$ with respect to the metric~$\rhos$.

Let us  recall the facts we need from the potential theory on a~compact Riemann surface $\SR$.
We will consider only the Riemann sphere~$\myh\CC$, that is,  we can assume that $\SR=\myh\CC$.
Let $\Lambda^j(\SR)$, $j=0,2,$ be the spaces of smooth $j$-forms on~$\SR$ with the topology of uniform convergence with all derivatives.
In particular, $\Lambda^0(\SR)=C^\infty(\SR)$ is the space of smooth functions on~$\SR$.
Let $\Lambda'^{j}(\SR):=(\Lambda^{2-j}(\SR))^*$ denote the dual space of $\Lambda^{2-j}(\SR)$, that is, (de Rham) currents   of degree $j$ with the  weak* topology on the dual space, see, for example, \cite{Chi06},
\cite{Rha56},~\cite{Chi15}.
In the general case, the operator $\ddc$ is defined on currents of degree~0 on~$\SR$ as
$\ddc:\Lambda'^{0}(\SR)\to \Lambda'^{2}(\SR)$ and acts by the rule  $\ddc T(\tau)=T(\ddc\tau)$, where
$T\in\Lambda'^{0}(\SR)$, and $\tau\in\Lambda^{0}(\SR)$ is an arbitrary test function. (As pointed out in \S\,\ref{s2},
on smooth functions~$\varphi$ in the local coordinate $\zeta = x +iy$ the operator~$\ddc$ acts as the Laplacian:
$\ddc\varphi = (\varphi_{xx}+\varphi_{yy}) \, dx\,dy=\Delta\varphi \,dx\,dy$.)
It is well known that the equation $\ddc T=\myt{T}$ on~$\SR$ is solvable for currents
  $\myt{T}\in\Lambda'^{2}(\SR) $ if and only if $\myt{T}(1)=0$, and if $\myt{T}(1)=0$, then $T$~is defined up to an additive constant
(because the solutions of the equation $\ddc T=0$ are harmonic functions on~$\SR$ by Weyl's lemma, and so, they are constants because $\SR$~is
compact). We will be interested in solutions of the equation $\ddc T=\myt{T}$ in the case when  $\myt T=\nu$ is
  a~neutral real-valued signed measure, that is, a~signed measure satisfying $\int_{\SR}\nu=0$.
  We denote by $\Meas_0(\SR)$ the space of all such signed measures. In this case, the current~$T$
is called the potential of the signed measure~$\nu$ and is denoted by $\myh\nu:= T=(\ddc)^{-1}\nu$ (the
potential $\myh\nu$ is defined up to an additive constant). The function~$\varphi$ on~$\SR$ is called
$\delta$-subharmonic if it can be locally represented as the difference of two subharmonic functions.
We denote by $\dsh(\SR)$ the space of all $\delta$-subharmonic functions on~$\SR$. It is well known  (see, for example,~\cite{Ch1})
that $\dsh(\SR)$ consists precisely of the potentials of all neutral signed measures on~$\SR$ and $\dsh(\SR)\in L^p(\SR)$ for any $p\in[1,\infty)$.
To get rid of the ambiguity of the operator $(\ddc)^{-1}$ on the space of neutral signed measures,
which is related to the possibility of adding a constant, one fixes a~continuous linear functional~$\phi$, say,
on the space $L^1(\SR)$ satisfying the condition $\phi(1)\ne 0$, and instead of the space~$\dsh(\SR)$ considers the space
\begin{equation}
\label{Pot_phi}
\Pot_\phi(\SR):=\{v\in\dsh(\SR):\phi(v)=0\}.
\end{equation}
On the space $\Pot_\phi(\SR)$, the operator $\ddc:\Pot_\phi(\SR)\to \Meas_0(\SR)$ is one-to-one.
For a~signed measure $\nu\in\Meas_0(\SR)$, we  denote by  $(\myh\nu)_\phi$ its potential $\myh\nu$ that lies in
the space $\Pot_\phi(\SR)$. Next, we will deal with the Riemann sphere~$\myh\CC$,
and as~$\phi$ we will use the functional defined by the area form~$d\sigma$ of the spherical metric~$\rhos$:
\begin{equation}
\label{ds(v)}
\phi(v)=d\sigma(v):=\int_{\myh\CC}v d\sigma,
\end{equation}
where $v\in\dsh(\myh\CC)\subset L^1(\myh\CC)$.

\subsection{Proof of Theorem~\ref{theorem1}}
\label{s5.2}

Theorem~\ref{theorem1} is the combination of Statements \ref{St_zero} and~\ref{St_lnP}, which we will prove in this section.

Recall that $S$ is the number of unknown  zeros of the divisor of $(\myt R_n)$ (see \eqref{div_tRn}).
We denote by $\alpha_{j_1,\dots,j_k}$ the number of zeros of the function  $M_{j_1,\dots,j_k}(\myt\zz)$~\eqref{MM} on the sheet $\TRS_{[k]}^{(01\dots k-1)}$, which is defined over $\myh \CC\setminus F_k$.
In what follows, the notation $M_{j_1,\dots,j_k}$ will mean the meromorphic function $M_{j_1,\dots,j_k}(\myt\zz)$  on~$\TRS_{[k]}$~\eqref{MM}.

\begin{Statement}
\label{St_zero}
Suppose that the surface $\TRS_{[k]}$, which is constructed from~$\pi$, is connected. Then,
for any neighbourhood~$V$ of the compact set $F_k$, there exists an $N=N(V)$ such that, for all $n>N$,
there are at most
$L_{j_1,\dots,j_k}:=S+\alpha_{j_1,\dots,j_k}$ zeros, taken with multiplicities, of the polynomial $P_{n;j_1,\dots,j_k}(z)$
outside the neighbourhood~$V$.
\end{Statement}

\begin{proof}
The functions  $M_{j_1,\dots,j_k}(\myt\zz)$~\eqref{MM}, which are
 meromorphic on $\TRS_{[k]}$, do not depend on~$n$,
 and hence, considering if necessary a~smaller neighborhood~$V$, we assume that on the set $\pi^{-1}(V\setminus F_k)$ these functions have neither zeros, nor poles.
 In the case $\infty\notin F_k$, we also require that $\infty\notin V$. We set
\begin{equation}
\label{delta}
\delta:=\frac{\dist(\partial V, F_k)}{2(2S+3)},
\end{equation}
where $\partial V$ is the boundary of~$V$.
Since  (see \cite{ChKoPaSu}, Attachment~1, Lemma~3) $F_k$ is a~piecewise analytic subset of
$\myh\CC$ without isolated points, for each~$n$ one  can find a~system of disjoint smooth contours $\Gamma_n$ bounding an open set $D_n$ such that $F_k\subset D_n\subset V$, and the following conditions hold true:
$\dist(\Gamma_n, F_k)\ge\delta$,
$\dist(\Gamma_n, \partial V)\ge\delta$,
$\dist(\Gamma_n, b_i)\ge\delta$,
$\dist(\Gamma_n, a_i(n))\ge\delta$, $i=1,\dots,S$,
where $b_i=\myt\pi(\myt\bb_i)$, $a_i(n)=\myt\pi(\myt\aa_i)$ are the projections of the poles and zeros of the remainder function $\myt R_n$ (see~\eqref{div_tRn}).

Let us get an upper estimate for the number of zeros of $P_{n;j_1,\dots,j_k}(z)$
in $\Omega_n:=\myh\CC\setminus\myo D_n$ (and hence, in $\myh\CC\setminus V$).
Since the functions $M_{j_1,\dots,j_k}(\myt\zz)$ and $\myt R_n(\myt\zz)$ are meromorphic on $\TRS_{[k]}$,   expression~\eqref{P_M}
extends to the whole $\myh \CC$ in the following way:
\begin{equation}
\label{P_M_1}
P_{n;j_1,\dots,j_k}(z)=\sum_{\myt\zz\in\myt\pi^{-1}(z)} M_{j_1,\dots,j_k}(\myt\zz)\myt R_n(\myt\zz).
\end{equation}
As pointed out in \S\,\ref{s4}, on the sheet $\TRS_{[k]}^{(01\dots k-1)}$ the
projection $\myt\pi:\TRS_{[k]}^{(01\dots k-1)}\to\myh\CC\setminus F_k$ is biholomorphic.
Hence, for $z\in \myh\CC\setminus F_k$, \eqref{P_M_1} is equivalent to
\begin{equation}
\label{P_Arg}
P_{n;j_1,\dots,j_k}(z)=M_{j_1,\dots,j_k}(\myt\zz^{(01\dots k-1)})\myt R_n(\myt\zz^{(01\dots k-1)})\left(1+
h_{n; j_1,\dots,j_k}(z)\right),
\end{equation}
where
\begin{equation}
\label{h_n}
h_{n; j_1,\dots,j_k}(z):=\sum_{\myt\zz\in\myt\pi^{-1}(z)\setminus \myt\zz^{(01\dots k-1)}}
\frac{M_{j_1,\dots,j_k}(\myt\zz)}{M_{j_1,\dots,j_k}(\myt\zz^{(01\dots k-1)})}\frac{\myt R_n(\myt\zz)}{\myt R_n(\myt\zz^{(01\dots k-1)})}.
\end{equation}

First of all, we show that $\lim\limits_{n\to\infty}\max\limits_{z\in\Gamma_n}|h_{n; j_1,\dots,j_k}(z)|= 0$.
By definitions of the function $\myt u$~\eqref{uu} and the functions $u_i$~\eqref{inequ_u},
for $\myt\zz\in\myt\pi^{-1}(z)\setminus \myt\zz^{(01\dots k-1)}$
we have  $\myt u(\myt \zz)-\myt u(\myt \zz^{(01\dots k-1)})\ge u_k(z)-u_{k-1}(z)$.
Therefore from \eqref{|tRn|_2} we conclude that, for $z\in \myh\CC\setminus F_k$,
\begin{equation}
\label{|h_n|}
\begin{aligned}
&|h_{n; j_1,\dots,j_k}(z)|\\
&\le\sum_{\myt\zz\in\myt\pi^{-1}(z)\setminus \myt\zz^{(01\dots k-1)}}  \left|\frac{M_{j_1,\dots,j_k}(\myt\zz)}{M_{j_1,\dots,j_k}(\myt\zz^{(01\dots k-1)})}\right|\frac{\psi_n(\myt\zz)}{\psi_n(\myt\zz^{(01\dots k-1)})}\cdot
e^{-n(u_k(z)-u_{k-1}(z))}.
\end{aligned}
\end{equation}
Let
\begin{equation}
K:=\{z\in V: \dist(z, F_k)\ge\delta/2, \dist(z, \partial V) \ge\delta/2\},
\end{equation}
where $\delta$ is defined by equality~\eqref{delta}. Note that  $\Gamma_n\subset K$ for all~$n$.
On~$\pi^{-1}(V\setminus F_k)$ the meromorphic functions $M_{j_1,\dots,j_k}(\myt\zz)$ have neither zeros nor poles, and so
\begin{equation}
\label{mod_M}
C_{j_1,\dots,j_k}:=\max\limits_{\myt\zz\in\pi^{-1}(K)\setminus \TRS_{[k]}^{(01\dots k-1)}}\left|\frac{M_{j_1,\dots,j_k}(\myt\zz)}{M_{j_1,\dots,j_k}(\myt\zz^{(01\dots k-1)})}\right|<\infty.
\end{equation}
The functions $u_i$ are continuous in~$\CC$  (see Appendix~1, Lemma~1 in~\cite{ChKoPaSu}), and besides,
for $k>1$ the function $u_k(z)-u_{k-1}(z)$ is continuous near~$\infty$, and for $k=1$ it tends to~$+\infty$ as $z\to\infty$. 
So,
since the compact set~$K$ does not intersect $F_k$, we have
\begin{equation}
\label{kappa}
\varkappa:=\min\limits_{z\in K}
\left( u_k(z)-u_{k-1}(z)\right)>0.
\end{equation}
In order to estimate  $\frac{\psi_n(\myt\zz)}{\psi_n(\myt\zz^{(01\dots k-1)})}$, where
$\psi_n(\myt\zz)=\exp\bigl\{\sum_{i=1}^S g(\myt\aa_i(n),\myt\bb_i;\myt\zz)\bigr\}$ (see \eqref{psi_n}),
we will derive an estimate for the functions $g(\myt\aa_i(n),\myt\bb_i;\myt\zz)$ themselves. For this purpose we use
Corollary~6 from~\cite{ChKoPaSu}, which in our setting states the following. 
For bipolar Green functions $g(\myt\qq,\myt\pp;\myt\zz)$ that
spherically  normalized on the sheet $\TRS_{[k]}^{(01\dots k-1)}$ and for an arbitrary $\delta'>0$,
there exists a~constant $C=C(\delta')$ (independent of~$\myt\qq$ and~$\myt\pp$) such that $|g(\myt\qq,\myt\pp;\myt\zz)|<C$
for all $\myt\zz\in\TRS_{[k]}$ such that $\dist_\rho(\myt\zz,\myt\qq)\ge \delta'$ and $\dist_\rho(\myt\zz,\myt\pp)\ge \delta'$.
The systems of contours $\Gamma_n$ are chosen so that  $\dist(\Gamma_n, a_i(n))\ge\delta$, $\dist(\Gamma_n, b_i)\ge\delta$,
and hence  (since for any conformal metric~$\rho$ on~$\RS$ there exists a~constant $C_\rho>0$ such that
$\dist_\rho(\myt\zz_1, \myt\zz_2)\ge C_\rho\dist(z_1, z_2)$ for arbitrary $\myt\zz_1, \myt\zz_2\in\TRS_{[k]}$),
it follows that there exists a~constant
$\myt C$ such that $|g(\myt\aa_i(n),\myt\bb_i;\myt\zz)|\le \myt C$ for $\myt\zz\in\pi^{-1}(\Gamma_n)$.
So, for $\myt\zz_1,\myt\zz_2\in\pi^{-1}(\Gamma_n)$ we have
\begin{equation}
\label{mod_psi}
\frac{\psi_n(\myt\zz_1)}{\psi_n(\myt\zz_2)}\le e^{2S\myt C}.
\end{equation}
Combining \eqref{mod_M},~\eqref{kappa}, \eqref{mod_psi} and taking into account that
$\TRS_{[k]}$ has ${m+1\choose k}$ sheets, from \eqref{|h_n|} we obtain that, for $z\in\Gamma_n$,
\begin{equation}
|h_{n; j_1,\dots,j_k}(z)|\le {m+1\choose k} C_{j_1,\dots,j_k}e^{2S\myt C}e^{-n\varkappa}.
\end{equation}
Since $\varkappa>0$, we have  $\lim\limits_{n\to\infty}\max\limits_{z\in\Gamma_n}|h_{n; j_1,\dots,j_k}(z)| = 0$.

Since on the sheet $\TRS_{[k]}^{(01\dots k-1)}$ the projection $\myt\pi:\TRS_{[k]}^{(01\dots k-1)}\to\myh\CC\setminus F_k$ is biholomorphic, the functions  $\myt R_n(\myt\zz^{(01\dots k-1)})$ and $M_{j_1,\dots,j_k}(\myt\zz^{(01\dots k-1)})$ in~\eqref{P_Arg} will
be considered as meromorphic functions of $z\in\myh\CC\setminus F_k$, that is,
for example, $\myt R_n(\myt\zz^{(01\dots k-1)})=\myt R_n \circ (\myt\pi|_{\TRS_{[k]}^{(01\dots k-1)}})^{-1}(z)$.
So,  all the functions involved in \eqref{P_Arg}, except for $1+h_{n; j_1,\dots,j_k}$, are meromorphic in $\myh\CC\setminus F_k$, and hence,
we also have  $1+h_{n; j_1,\dots,j_k}(z)\in\MM(\myh\CC\setminus F_k)$.
Since $\lim\limits_{n\to\infty}\max\limits_{z\in\Gamma_n}|h_{n; j_1,\dots,j_k}(z)|= 0$, there exists an $N=N(V)$ such that,
for all $n>N$ we have $|h_{n; j_1,\dots,j_k}(z)|<1/2$ for $z\in\Gamma_n$.
Further, we assume that $n>N$.
Then $1+h_{n; j_1,\dots,j_k}(z)$ has no zeros on $\Gamma_n$.
Besides, by the choice of the system of contours $\Gamma_n$, the functions $\myt R_n(\myt\zz^{(01\dots k-1)})$ and $M_{j_1,\dots,j_k}(\myt\zz^{(01\dots k-1)})$
have neither zeros nor poles on $\Gamma_n$.
Hence all the functions from the right-hand side of~\eqref{P_Arg} are meromorphic  in~$\Omega_n$ and have no
zeros on $\Gamma_n = \partial\Omega_n$. Therefore,
the number of zeros of $P_{n;j_1,\dots,j_k}$ in~$\Omega_n$ can be evaluated using the argument principle.
We choose the orientation of each contour from~$\Gamma_n$ in such a way that $\Gamma_n$ is positively oriented with respect to~$\Omega_n$.

Assume first that  $\infty\notin F_k$. (Recall that in this case $\infty\notin V$, and hence, $\infty\in\Omega_n$.)
Then the number of zeros of the polynomials $P_{n;j_1,\dots,j_k}$ in~$\Omega_n$ is equal to
\begin{equation}
\label{Zero_P}
\deg P_{n;j_1,\dots,j_k}+\frac{1}{2\pi}\mDelta_{z\in \Gamma_n}\arg P_{n;j_1,\dots,j_k}(z).
\end{equation}
Since  $|h_{n; j_1,\dots,j_k}(z)|<1/2$ on $\Gamma_n$, we have  $\mDelta\limits_{z\in \Gamma_n}\arg (1+h_{n; j_1,\dots,j_k}(z))=0$. Therefore
from \eqref{P_Arg} we obtain
\begin{equation}
\label{arg_P}
\mDelta_{z\in \Gamma_n}\arg P_{n;j_1,\dots,j_k}(z)=\mDelta_{z\in \Gamma_n}\arg \myt R_n(\myt\zz^{(01\dots k-1)})+\mDelta_{z\in \Gamma_n}\arg M_{j_1,\dots,j_k}(\myt\zz^{(01\dots k-1)}).
\end{equation}
By the argument principle, $\frac{1}{2\pi}\mDelta_{z\in \Gamma_n}\arg M_{j_1,\dots,j_k}(\myt\zz^{(01\dots k-1)})$ equals
the  difference of the  number of  zeros and poles (taken with multiplicities)
of the function  $M_{j_1,\dots,j_k}(\myt\zz^{(01\dots k-1)})$ for $z\in \Omega_n$. Hence
\begin{equation}
\label{arg_M}
\frac{1}{2\pi}\mDelta_{z\in \Gamma_n}\arg M_{j_1,\dots,j_k}(\myt\zz^{(01\dots k-1)})\le \alpha_{j_1,\dots,j_k}.
\end{equation}
Since $\infty\notin F_k$, the point $\myt{\pmb\infty}^{(01\dots k-1)}$ is not a~critical point for $\myt\pi$.
Hence, from the form of the  divisor of the function $\myt R_n$ (see \eqref{div_tRn}) we
conclude that the function $\myt R_n(\myt\zz^{(01\dots k-1)})$ has a~pole of order $(m+1-k)n$ at $\myt{\pmb\infty}^{(01\dots k-1)}$ 
(not taken into accout poles $\myt\bb_j$ and free zeros $\myt\aa_j(n)$, the number of which at 
$\myt{\pmb\infty}^{(01\dots k-1)}$ is not greater than $S$). Consequently,
\begin{equation}
\label{arg_R_1}
\frac{1}{2\pi}\mDelta_{z\in \Gamma_n}\arg \myt R_n(\myt\zz^{(01\dots k-1)})\le -(m+1-k)n+S.
\end{equation}
We have  $\deg P_{n;j_1,\dots,j_k}\le (m+1-k)n$, and hence, combining \eqref{Zero_P},~\eqref{arg_P} and \eqref{arg_R_1},
we see that the number of zeros of the polynomials $P_{n;j_1,\dots,j_k}$ in~$\Omega_n$ (for $n>N$) is at most  $L_{j_1,\dots,j_k}:=S+\alpha_{j_1,\dots,j_k}$.

If $\infty\in F_k$, then we always have $\infty\notin\Omega_n$. Therefore,  the number of zeros of the
polynomials $P_{n;j_1,\dots,j_k}$ in~$\Omega_n$ is equal to
$\frac{1}{2\pi}\mDelta\limits_{z\in \Gamma_n}\arg P_{n;j_1,\dots,j_k}(z)$.
Proceeding as in the case $\infty\notin F$, we conclude that
estimate~\eqref{arg_M} for $\mDelta\limits_{z\in \Gamma_n}\arg M_{j_1,\dots,j_k}(\myt\zz^{(01\dots k-1)})$ again holds true,
and from \eqref{div_tRn} we obtain that
$\frac{1}{2\pi}\mDelta\limits_{z\in \Gamma_n}\arg \myt R_n(\myt\zz^{(01\dots k-1)})\le S$.
As a~result, we again get that the number of zeros of the polynomials $P_{n;j_1,\dots,j_k}$ in~$\Omega_n$ (for  $n>N$) is at most
 $L_{j_1,\dots,j_k}:=S+\alpha_{j_1,\dots,j_k}$.
\end{proof}

Recall that the functions $h_{n; j_1,\dots,j_k}(z)$ are defined in \eqref{h_n}.

\begin{Lemma}
\label{L_1}
Suppose that the surface $\TRS_{[k]}$, which is constructed from~ $\pi$, is con\-nect\-ed.
Then, for any neighbourhood~$V$ of the compact set~$F_k$, there exists an $N=N(V)$
such that, for all $n>N$, the number of zeros and poles (taken with multiplicities) of the function $1+h_{n; j_1,\dots,j_k}(z)$ in $\myh\CC\setminus V$
is at most $L_{j_1,\dots,j_k}$.
\end{Lemma}
\begin{proof}
By definition~\eqref{h_n},
\begin{equation}
\label{L_h}
1+h_{n; j_1,\dots,j_k}(z)=\frac{P_{n;j_1,\dots,j_k}(z)}{\myt R_n(\zz^{(01\dots k-1)})M_{j_1,\dots,j_k}(\myt\zz^{(01\dots k-1)})},
\end{equation}
where, as in the proof of Statement~\ref{St_zero}, the functions $\myt R_n(\myt\zz^{(01\dots k-1)})$ and $M_{j_1,\dots,j_k}(\myt\zz^{(01\dots k-1)})$
are understood as meromorphic functions of  $z\in\myh\CC\setminus F_k$.
Hence  $h_{n; j_1,\dots,j_k}(z)\in\MM(\myh\CC\setminus F_k)$.
Taking into account the form of the divisor of $\myt R_n$~\eqref{div_tRn} and
since $\deg P_{n; j_1,\dots, j_k}\le (m+1-k)n$, we conclude that the number of poles of the function  $1+h_{n; j_1,\dots,j_k}(z)$ in $\myh\CC\setminus F_k$
is at most $S+\alpha_{j_1,\dots,j_k}=L_{j_1,\dots,j_k}$. Consequently,  their number in $\myh\CC\setminus V$ is also at most $L_{j_1,\dots,j_k}$.

For each $n$, we choose the same system of contours $\Gamma_n$ as in the proof of
Statement~\ref{St_zero}. In particular, $\Gamma_n$ bounds the open set $D_n$ such that $F_k\subset D_n\subset V$.
In the proof of Statement~\ref{St_zero} it was shown that
$\lim\limits_{n\to\infty}\max\limits_{z\in\Gamma_n}|h_{n; j_1,\dots,j_k}(z)|= 0$. We choose an~$N$ such that, for all $n>N$,
$|h_{n; j_1,\dots,j_k}(z)|<1/2$ for $z\in\Gamma_n$.
Further, we assume that $n>N$.
Since $|h_{n; j_1,\dots,j_k}(z)|<1/2$ on $\Gamma_n$, by the argument principle the number of zeros of the function  $1+h_{n; j_1,\dots,j_k}$
in the open set $\Omega_n:=\myh\CC\setminus\myo D_n$ is equal to the number of its poles in the same set.
We have $\myh\CC\setminus V\subset\Omega_n$, and hence the number of zeros
of the function  $1+h_{n; j_1,\dots,j_k}(z)$ for $z\in\myh\CC\setminus V$ is not greater than the number of its
poles in $\myh\CC\setminus F_k$, that is,  $L_{j_1,\dots,j_k}$.
\end{proof}

Recall that the functions $u_i$, which define the Nuttall partition \eqref{inequ_u}, are already defined uniquely,
because the normalization of the function $\myt u$~\eqref{uu} is fixed in \eqref{norm_g}:
\begin{equation}
\label{norm_sum_u}
\int _{\myh\CC}\sum_{s=0}^{k-1}u_s(z)d\sigma(z)\equiv
\int _{\myh\CC\setminus F_k}\myt u(\myt\zz^{(01\dots k-1)})d\sigma(z)=0.
\end{equation}

\begin{Statement}
\label{St_lnP}
Suppose that the  surface $\TRS_{[k]}$, which is constructed from~$\pi$, is connected. Then, for an arbitrary $p\in[1,\infty)$, as $n\to\infty$,
\begin{equation}
\label{lnP}
\frac{1}{n}\log|P_{n;j_1,\dots,j_k}^*(z)|\to -\sum_{s=0}^{k-1}u_s(z)
\quad \text{in}\quad L^p(\myh\CC),
\end{equation}
\begin{equation}
\label{ddc_Pn}
\frac{1}{n}\ddc \log|P_{n;j_1,\dots,j_k}(z)|\xrightarrow{*} -\ddc \left(\sum_{s=0}^{k-1}u_s(z)\right) \quad\text{in}\quad C(\myh\CC)^*.
\end{equation}
\end{Statement}

\begin{proof}
We fix the index set $j_1,\dots, j_k$.
Using \eqref{P_Arg} and taking into account \eqref{|tRn|_2}, we see that,  for $z\in\myh\CC\setminus F_k$,
\begin{equation}
\label{lnPn*_1}
\begin{aligned}
&\frac{1}{n}\log|P_{n;j_1,\dots,j_k}(z)|=-\sum_{s=0}^{k-1}u_s(z) +\\
&\frac{1}{n}\log\left\{\psi_n(\myt\zz^{(01\dots k-1)})\cdot |M_{j_1,\dots,j_k}(\myt\zz^{(01\dots k-1)})|\cdot |1+h_{n; j_1,\dots,j_k}(z)|\right\},
\end{aligned}
\end{equation}
where $h_{n; j_1,\dots,j_k}$ is defined in \eqref{h_n}.
(Here  $\psi_n(\myt\zz^{(01\dots k-1)})$ and $M_{j_1,\dots,j_k}(\myt\zz^{(01\dots k-1)})$ are again understood  as
meromorphic functions of  $z\in\myh\CC\setminus F_k$.)
Since $F_k$ is a~piecewise analytic subset of~$\myh\CC$ (see \cite{ChKoPaSu}, Attachment~1, Lemma~3),
we have $\sigma(F_k)=0$. Hence equality~\eqref{lnPn*_1} can be understood as the equality of two elements of $L^p(\myh\CC)$.
Next we assume that $p\in[1,\infty)$ is fixed.

Since (see \eqref{norm_g}) we spherically normalized all bipolar Green functions  $g(\myt\qq,\myt\pp;\myt\zz)$
on the sheet $\TRS_{[k]}^{(01\dots k-1)}$ of the compact Riemann surface $\TRS_{[k]}$, which is connected by the assumption,
we have (see \cite{ChKoPaSu}, Attachment~2, Corollary~5) that their norms in the space $L_p(\TRS_{[k]})$
are uniformly bounded by some constant~$C_1$. (Note that in~\cite{ChKoPaSu}
this corollary was formulated only for $p\in(1,\infty)$, but its proof
also holds for $p=1$.
Moreover, this is immaterial for us, because on any compact set the convergence in the space $L^p$ 
implies the convergence in all $L^q$ with $1\le q<p$.)
Since the surface $\TRS_{[k]}$ is compact, for any function $f\in L^p(\TRS_{[k]})$ we have
$\|f(\myt\zz^{(01\dots k-1)})\|_{L^p(\myh\CC)}\le C_2\|f\|_{L^p(\RS)}$, where $C_2:=\max\limits_{\myt\zz\in\TRS}\left(\frac{d\sigma(z)}{d\sigma_\rho(\myt\zz)}\right)^{1/p}<\infty$.
Therefore from definition \eqref{psi_n} we get
\begin{equation}
\label{||psi||}
\|\log\psi_n(\myt\zz^{(01\dots k-1)})\|_{L^p(\myh\CC)}\le
C_2\|\log\psi_n(\myt\zz)\|_{L^p(\TRS_{[k]})}\le
C_2 C_1 S.
\end{equation}
Consequently, $\frac{1}{n}\log\psi_n(\myt\zz^{(01\dots k-1)})\to 0$ in $L^p(\myh\CC)$.
Since $M_{j_1,\dots, j_k}(\myt\zz)\in\MM(\TRS_{[k]})$, we have
 $\log|M_{j_1,\dots, j_k}(\myt\zz)|\in L^p(\TRS_{[k]})$. This implies that
 $\log|M_{j_1,\dots, j_k}(\myt\zz^{(01\dots k-1)})|\in L^p(\myh\CC)$. Therefore,  $\frac{1}{n}\log |M_{j_1,\dots, j_k}(\myt\zz^{(01\dots k-1)})|\to 0$ in $L^p(\myh\CC)$.
Let us now show that  $\frac{1}{n}\log|1+h_{n;j_1\dots,j_k}(z)|\to 0$ in $L^p_{\loc}(\myh\CC\setminus F_k)$.
In view of the above, this together with \eqref{lnPn*_1} will imply that $\frac{1}{n}\log|P_{n;j_1,\dots j_k}|\to -\sum_{s=0}^{k-1}u_s(z)$ in $L^p_{\loc}(\myh\CC\setminus F_k)$.

So, fix a neighborhood $V$ of the compact set $F_k$. Let us show that
$\frac{1}{n}\log|1+h_{n; j_1,\dots, j_k}|\to 0$ in $L^p(\myh\CC\setminus V)$.
Passing if necessary to a~smaller neighborhood, we assume that the function $M_{j_1,\dots, j_k}(\myt\zz)$
has neither zeros nor poles on the set $\myt\pi^{-1}(V\setminus F_k)$, and if $\infty\notin F_k$, we assume that $\infty\notin V$.
We set
\begin{equation}
\label{delta_2}
\delta:=\frac{\dist(\partial V, F_k)}{2(2S+2L_{j_1,\dots, j_k}+3)},
\end{equation}
where $L_{j_1,\dots,j_k}:=S+\alpha_{j_1,\dots,j_k}$, $S$ is the number of unknown zeros in the divisor
of $(\myt R_n)$~\eqref{div_tRn},  and $\alpha_{j_1,\dots,j_k}$ is the number of zeros
of the function $M_{j_1,\dots,j_k}(\myt\zz)$ on the sheet $\TRS_{[k]}^{(01\dots k-1}$.
Let $V_\delta:=\{z\in\myh\CC:\dist(z, F_k)<\delta\}$.
By Lemma~\ref{L_1}, there exists an $N$ such that, for all $n>N$, the number of zeros and poles (taken with multiplicities)
of the function $1+h_{n; j_1,\dots, j_k}(z)$ in $\myh\CC\setminus V_\delta$ is at most $L_{j_1,\dots,j_k}$.
Further we assume that $n>N$.
Let  $q_1(n),\dots,q_{l(n)}(n)$ be the zeros  (taken with multiplicities)
of the function $(1+h_{n; j_1,\dots, j_k})$ in $\myh\CC\setminus V_\delta$ an $p_1(n),\dots,p_{l'(n)}(n)$ be the poles (also taken with multiplicities) of it.
Then $l(n), l'(n)\le L_{j_1,\dots, j_k}$.
We set $\myt\qq_s(n):=\myt\pi^{-1}(q_s(n))\cap\TRS_{[k]}^{(01\dots k-1)}$ and $\myt\pp_s(n):=\myt\pi^{-1}(p_s(n))\cap\TRS_{[k]}^{(01\dots k-1)}$.
We fix a~point $\myt\zz^*\in\partial\TRS_{[k]}^{(01\dots k-1)}$, and for $\myt\zz\in\TRS_{[k]}$, define the function
\begin{equation}
\label{psi_tilde}
\psi_{n;j_1,\dots j_k}(\zz):=\exp\biggl\{\sum_{s=1}^{L_{j_1,\dots j_k}} g(\myt\qq_s(n),\myt\zz^*;\myt\zz)+
\sum_{s=1}^{L_{j_1,\dots j_k}} g(\myt\zz^*, \myt\pp_s(n);\myt\zz) \biggr\},
\end{equation}
where bipolar Green functions   $g(\myt\qq,\myt\pp;\zz)$~\eqref{g_as} are spherically normalized
on the sheet $\TRS_{[k]}^{(01\dots k-1)}$ (see~\eqref{norm_g}).
In~\eqref{psi_tilde}, we assume that  $l(n)=l'(n)=L_{j_1,\dots j_k}$,
complementing, if necessary, the collections $\{\myt \qq_s(n)\}_{s=1}^{l(n)}$ and $\{\myt \pp_s(n)\}_{s=1}^{l'(n)}$
with the point $\zz^*$ (taking $L_{j_1,\dots j_k}-l(n)$  and $L_{j_1,\dots j_k}-l'(n)$ times, respectively);
we also complement the collections $\{q_s(n)\}$ and $\{p_s(n)\}$ with the point $z^*=\myt\pi(\zz^*)$ (taking it the same number of times).
Proceeding as in the derivation of the uniform estimate for the functions $\psi_n$~\eqref{||psi||}, we find that
\begin{equation*}
\|\log\psi_{n;j_1,\dots, j_k}(\myt\zz^{(01\dots k-1)})\|_{L^p(\myh\CC)}\le
2C_2 C_1 L_{j_1,\dots j_k}.
\end{equation*}
Therefore $\frac{1}{n}\log\psi_{n;j_1,\dots, j_k}(\myt\zz^{(01\dots k-1)})\to 0$ in $L^p(\myh\CC)$.
So, it remains to verify that
\begin{equation}
\label{h/psi}
\frac{1}{n}\log\frac{|1+h_{n;j_1,\dots, j_k}(z)|}{\psi_{n;j_1,\dots, j_k}(\myt\zz^{(01\dots k-1)})}
\longrightarrow 0 \quad \text{in}\quad L^p(\myh\CC\setminus V).
\end{equation}

We will prove that the function $\log\frac{|1+h_{n;j_1,\dots, j_k}(z)|}{\psi_{n;j_1,\dots, j_k}(\myt\zz^{(01\dots k-1)})}$
is uniformly bounded   on the set $\myh\CC\setminus V$. Of course, this will imply \eqref{h/psi}.
As pointed out above (see~\eqref{L_h}), the function $1+h_{n; j_1,\dots,j_k}$ is meromorphic  in~$\myh\CC\setminus F_k$. Therefore,
$\log|1+h_{n; j_1,\dots,j_k}|$ is a~harmonic function in $\myh\CC\setminus V_\delta$, except the
points $q_1(n),\dots,q_{l(n)}(n)$ and $p_1(n),\dots,p_{l'(n)}(n)$, where it has the corresponding logarithmic singularities.
By construction (see \eqref{psi_tilde}) the function $\log\psi_{n;j_1,\dots j_k}$ is also harmonic in this domain and has
the same logarithmic singularities as $\log|1+h_{n; j_1,\dots,j_k}|$ at the points $q_1(n),\dots,q_{l(n)}(n)$ and $p_1(n),\dots,p_{l'(n)}(n)$,
and hence the function $\log\frac{|1+h_{n;j_1,\dots, j_k}(z)|}{\psi_{n;j_1,\dots, j_k}(\myt\zz^{(01\dots k-1)})}$ is harmonic in $\myh\CC\setminus V_\delta$.
By the choice of~$\delta$ (see \eqref{delta_2}), for each~$n$ one can find a~system of disjoint smooth contours~$\Gamma_n$
bounding an open set $D_n$ such that $F_k\subset D_n\subset V$,  and the following conditions hold true:
$\dist(\Gamma_n, F_k)\ge\delta$,
$\dist(\Gamma_n, \partial V)\ge\delta$,
$\dist(\Gamma_n, q_s(n))\ge\delta$,
$\dist(\Gamma_n, p_{s}(n))\ge\delta$, $s=1,\dots,L_{j_1,\dots j_k}$,
$\dist(\Gamma_n, a_i)\ge\delta$,
$\dist(\Gamma_n, b_i(n))\ge\delta$, $i=1,\dots,S$,
where $a_i(n)=\myt\pi(\aa_i(n))$, $b_i=\myt\pi(\bb_i)$ are the projections of the zeros and poles of the remainder function $\myt R_n$ (see \eqref{div_tRn}).
In particular, we have
$\dist(\Gamma_n, F_k)\ge\delta$,
$\dist(\Gamma_n, \partial V)\ge\delta$,
$\dist(\Gamma_n, b_i)\ge\delta$,
$\dist(\Gamma_n, a_i(n))\ge\delta$, and so $\lim\limits_{n\to\infty}\max\limits_{z\in\Gamma_n}|h_{n;j_1,\dots, j_k}(z)|= 0$ (see the derivation of the analogous property in the proof of Statement~\ref{St_zero}).
Therefore,  there exists an $N'$ such that $|h_{n;j_1,\dots, j_k}(z)|<1/2$ on $\Gamma_n$ for all $n>N'$.
Further, we assume that $n>N'>N$. Hence, for $z\in\Gamma_n$,
\begin{equation}
\label{h_bound}
1/2<|1+h_{n;j_1,\dots, j_k}(z)|<3/2.
\end{equation}
We have
$\dist(\Gamma_n, \myt p_s(n))\ge\delta$,
$\dist(\Gamma_n, \myt q_{s}(n))\ge\delta$, and
$\dist(\Gamma_n, z^*)\ge\delta$, and so, proceeding as in the derivation of estimate \eqref{mod_psi} in the proof of
Statement \ref{St_zero}, we obtain that there exists a~constant
$\myt C$ such that
$|g(\myt\qq_s(n), \myt\zz^*;\myt\zz)|\le\myt C$ and $|g(\myt\zz^*, \myt\pp_s(n);\myt\zz)|\le\myt C$
for $\myt\zz\in\pi^{-1}(\Gamma_n)$.
So,  from \eqref{psi_tilde} we get, for $z\in\Gamma_n$,
\begin{equation}
\label{psi_tilde_bound}
|\log\psi_{n;j_1,\dots, j_k}(\myt\zz^{(01\dots k-1)})|\le 2L_{j_1,\dots, j_k} \myt C.
\end{equation}
Combining \eqref{h_bound} and \eqref{psi_tilde_bound}, we find that, for $z\in\Gamma_n$,
\begin{equation}
\label{log_h/psi}
\left|\log\frac{|1+h_{n;j_1,\dots, j_k}(z)|}{\psi_{n;j_1,\dots, j_k}(\myt\zz^{(01\dots k-1)})}\right|\le
2L_{j_1,\dots, j_k} \myt C + 1.
\end{equation}
We set  $\Omega_n:=\myh\CC\setminus \myo D_n$.
Since the  function $\log\frac{|1+h_{n;j_1,\dots, j_k}(z)|}{\psi_{n;j_1,\dots, j_k}(\myt\zz^{(01\dots k-1)})}$ is harmonic in $\myh\CC\setminus V_\delta\supset\Omega_n$ and $\Gamma_n=\partial \Omega_n$, from the maximum principle it follows that estimate~\eqref{log_h/psi} holds
in the whole domain  $\Omega_n$, and therefore, in $\myh\CC\setminus V\subset\Omega_n$.
So, we have shown that $\frac{1}{n}\log|P_{n;j_1,\dots j_k}|\to -\sum_{s=0}^{k-1}u_s(z)$ in $L^p_{\loc}(\myh\CC\setminus F_k)$.

To prove the statement, we show that from any subsequence $\{P^*_{n; j_1,\dots,j_k}\}$, $n\in\Lambda$,
of polynomials $P^*_{n; j_1,\dots,j_k}$ one can choose a~subsequence $\{P^*_{n; j_1,\dots,j_k}\}$, $n\in\Lambda'\subset\Lambda$
satisfying \eqref{lnP} and~\eqref{ddc_Pn}.
By the Poincar\'e---Lelong formula  (see, for example, \cite{Chi06}), we have
\begin{equation}
\label{P--L}
\mu_{n; j_1,\dots,j_k}:=\ddc\log |P^*_{n; j_1,\dots,j_k}|=2\pi\Bigl(\sum\limits_{z:P^*_{n; j_1,\dots,j_k}(z)=0}\delta_z
-\deg P^*_{n; j_1,\dots,j_k}\cdot\delta_\infty\Bigr),
\end{equation}
where $\delta_x$ is the delta-measure at the point $x\in\myh\CC$.
Since  $\deg P^*_{n; j_1,\dots,j_k}\le (m+1-k)n$, we have $\|\frac{1}{n}\mu_{n; j_1,\dots,j_k}\|_{C(\myh\CC)^*}\le 4\pi(m+1-k)$.
By the Banach--Alaoglu theorem  (on the compactness of a~norm-closed ball of the dual space in weak* topology),
the sequence $\{\frac{1}{n}\mu_{n; j_1,\dots,j_k}\}$, $n\in\Lambda$, has a~subsequence $\{\frac{1}{n}\mu_{n; j_1,\dots,j_k}\}$, $n\in\Lambda'\subset\Lambda$,
which  weak* converges to some signed measure  $\mu_{j_1,\dots,j_k}\in C(\myh\CC)^*$. Let us show that  $\Lambda'$
satisfies \eqref{lnP} and \eqref{ddc_Pn}. Consider the space of potentials $\Pot_\phi(\myh\CC)$ (see \eqref{Pot_phi}), where
the functional~$\phi$ is defined by the area form $d\sigma$ (see~\eqref{ds(v)}).
The measure $d\sigma$ has smooth density with respect to the Lebesgue measure in any coordinate neighbourhood,
and hence, clearly, all its local potentials are continuous. So, we can apply the corollary of the lemma in
\S2.3 in~\cite{Ch1}, which states that if a~measure defining the functional~$\phi$ has continuous local potentials, then
the weak convergence of measures  implies the convergence of their potentials from $\Pot_\phi$ in all spaces
$L^p$ with  $p\in[1,\infty)$. Hence, $\frac{1}{n}(\myh\mu_{n; j_1,\dots,j_k})_\phi\to(\myh\mu_{j_1,\dots,j_k})_\phi$ for
$n\in\Lambda'$ in $L^p(\myh\CC)$ for all $p\in[1,\infty)$.
Since the  polynomials $P^*_{n; j_1,\dots,j_k}$ are spherically normalized (see \eqref{P^*}),
we have $P^*_{n; j_1,\dots,j_k}\equiv(\myh\mu_{n; j_1,\dots,j_k})_\phi$. Consequently, as $n\to\infty$, $n\in\Lambda'$,
\begin{equation}
\label{Pn*_to}
\frac{1}{n}\log |P^*_{n; j_1,\dots,j_k}|\to
(\myh\mu_{j_1,\dots,j_k})_\phi
\quad\text{in}\quad L^p(\myh\CC)
\end{equation}
for all $p\in[1,\infty)$.

Using \eqref{Pn*_to} and since $\frac{1}{n}\ddc\log |P^*_{n; j_1,\dots,j_k}|\xrightarrow{*}\mu_{j_1,\dots,j_k}$ as $n\in\Lambda'$  by construction, we see that it remains to show that $(\myh\mu_{j_1,\dots,j_k})_\phi=-\sum_{s=0}^{k-1}u_s$ in $L^p(\myh\CC)$. Indeed, we have
 $\frac{1}{n}\log|P_{n; j_1,\dots,j_k}|\to -\sum_{s=0}^{k-1}u_s$ in $L^p_{\loc}(\myh\CC\setminus F_k)$ and $P^*_{n; j_1,\dots,j_k}=c_{n; j_1,\dots,j_k}P_{n; j_1,\dots,j_k}$,
where $c_{n; j_1,\dots,j_k}>0$ are some constants, and hence from \eqref{Pn*_to} we find that, as  $n\to\infty$, $n\in\Lambda'$,
\begin{equation*}
\frac{1}{n}\log c_{n; j_1,\dots,j_k}
\to (\myh\mu_{j_1,\dots,j_k})_\phi+\sum_{s=0}^{k-1}u_s
\quad \text{in} \quad L^p_{\loc}(\myh\CC\setminus F_k).
\end{equation*}
Since $c_{n; j_1,\dots,j_k}$ are constants, they can converge only to a~constant. Therefore,
$(\myh\mu_{j_1,\dots,j_k})_\phi=-\sum_{s=0}^{k-1}u_s+\const$ in $L^p(\myh\CC)$.
On the other hand, the function $-\sum_{s=0}^{k-1}u_s$ is spherically normalized (see \eqref{norm_sum_u}), and so we have 
$\phi((\myh\mu_{j_1,\dots,j_k})_\phi)=\phi(-\sum_{s=0}^{k-1}u_s)=0$.
So, $\const=0$, that is,  $(\myh\mu_{j_1,\dots,j_k})_\phi=-\sum_{s=0}^{k-1}u_s$.
\end{proof}

\subsection{Proof of Theorem~\ref{theorem2}}
\label{s5.3}

Theorem \ref{theorem2} is precisely Corollary~\ref{Cor_P/P} (see also Re\-mark~\ref{t2_rem}), which  will be proved
in this section.

Let $\myt\ww_1, \dots \myt\ww_W$ be all zeros and poles  (taken without multiplicities)  of all functions $M_{j_1,\dots,j_k}(\myt\zz)$,
$0\le j_1<\dots<j_k\le m$, and let  $w_i=\myt\pi(\myt\ww_i)$, $i=1,\dots,W$, be their projections.
Recall that $a_i(n)=\myt\pi(\myt\aa_i(n))$, $b_i=\myt\pi(\myt\bb_i)$, $i=1,\dots,S$, are the projections of the zeros and poles
of the remainder function $R_n$ (see \eqref{div_tRn}).
For any $\varepsilon>0$ and any point $z^*\in\myh\CC$, we denote by
$O_{z^*}^\varepsilon:=\{z:\dist(z, z^*)<\varepsilon\}$
the disc in the spherical metric with  centre at~$z^*$ of radius~$\varepsilon$.
For any compact set $K\subset\myh\CC$ and $\varepsilon>0$, we set
\begin{equation}
K^\varepsilon(n):=K\setminus\left(\bigcup_{i=1}^{W}O_{w_i}^\varepsilon
\cup \bigcup_{i=1}^{S}O_{a_i(n)}^\varepsilon
\cup \bigcup_{i=1}^{S}O_{b_i}^\varepsilon\right).
\end{equation}

\begin{Statement}
\label{St_P/P}
Assume that the  surface $\TRS_{[k]}$, which is constructed from~$\pi$, is connected.
Then, for any compact set $K\subset\myh\CC\setminus F_k$ and arbitrary $\varepsilon>0$,
\begin{equation}
\label{P/P}
\lim_{n\to\infty}\max_{z\in K^\varepsilon(n)}\left|\frac{P_{n;j_1,\dots,j_k}(z)}{P_{n;i_1,\dots,i_k}(z)}
-\frac{M_{j_1,\dots,j_k}(\myt\zz^{(01\dots k-1)})}{M_{i_1,\dots,_k}(\myt\zz^{(01\dots k-1)})}\right|=0.
\end{equation}
Moreover, for the rate of convergence we have the following estimate
\begin{equation}
\label{P/P^}
\varlimsup_{n\to\infty}\max_{z\in K^\varepsilon(n)}\left(\left|\frac{P_{n;j_1,\dots,j_k}(z)}{P_{n;i_1,\dots,i_k}(z)}
-\frac{M_{j_1,\dots,j_k}(\myt\zz^{(01\dots k-1)})}{M_{i_1,\dots,_k}(\myt\zz^{(01\dots k-1)})}\right|^{1/n}\cdot
e^{u_{k}(z)-u_{k-1}(z)}\right)\le 1.
\end{equation}
\end{Statement}

\begin{proof}
We fix a compact set $K\subset\myh\CC\setminus F_k$ and $\varepsilon>0$.
From the expression for  $P_{n;j_1,\dots,j_k}$~\eqref{P_Arg} we see that, for $z\in K$,
\begin{equation}
\label{Pn/Pn}
\frac{P_{n;j_1,\dots,j_k}(z)}{P_{n;i_1,\dots,i_k}(z)}=
\frac{M_{j_1,\dots,j_k}(\myt\zz^{(01\dots k-1)})}{M_{i_1,\dots,_k}(\myt\zz^{(01\dots k-1)})}
\cdot\frac{1+h_{n; j_1\dots j_k}(z)}{1+h_{n;i_1\dots i_k}(z)},
\end{equation}
where $h_{n; j_1,\dots,j_k}$ are defined  in \eqref{h_n}.
Let us show that  $\lim\limits_{n\to\infty}\max\limits_{z\in K^\varepsilon(n)}|h_{n; j_1\dots j_k}(z)|=0$.
Of course, this will imply~\eqref{P/P}.
We will proceed as in the proof of Statement~\ref{St_zero} in the derivation of the
property  $\lim\limits_{n\to\infty}\max\limits_{z\in \Gamma_n}|h_{n;j_1,\dots, j_k}(z)|=0$. So,
for $z\in K$, for $h_{n; j_1,\dots,j_k}$ we use estimate~\eqref{|h_n|}
\begin{equation*}
\begin{aligned}
&|h_{n; j_1,\dots,j_k}(z)|\\
&\le \sum_{\myt\zz\in\myt\pi^{-1}(z)\setminus \myt\zz^{(01\dots k-1)}}  \left|\frac{M_{j_1,\dots,j_k}(\myt\zz)}{M_{j_1,\dots,j_k}(\myt\zz^{(01\dots k-1)})}\right|\frac{\psi_n(\myt\zz)}{\psi_n(\myt\zz^{(01\dots k-1)})}\cdot e^{-n(u_k(z)-u_{k-1}(z))}.
\end{aligned}
\end{equation*}
Let $K_1:=K\setminus\bigcup\limits_{k=1}^W O^\varepsilon_{w_k}$. (In particular, $K^\varepsilon(n)\subset K_1$.) Then
\begin{equation*}
C_{j_1\dots j_k}:=\max\limits_{\myt\zz\in\myt\pi^{-1}(K_1)}
\left|\frac{M_{j_1,\dots,j_k}(\myt\zz)}{M_{j_1,\dots,j_k}(\myt\zz^{(01\dots k-1)})}\right|<\infty.
\end{equation*}
We have $\dist(K^\varepsilon(n), a_i(n))\ge\varepsilon$, $\dist(K^\varepsilon(n), b_i)\ge\varepsilon$, $i=1,\dots,S$,
and so, proceeding as in the derivation of estimate \eqref{mod_psi} in the proof of Statement~\ref{St_zero}, we find that
there exists a~constant  $\myt C_1=\myt C_1(\varepsilon)$ such that, for $\myt\zz_1,\myt\zz_2\in\pi^{-1}(K^\varepsilon(n))$,
\begin{equation*}
\frac{\psi_n(\myt\zz_1)}{\psi_n(\myt\zz_2)}\le e^{2S\myt C_1}.
\end{equation*}
Consequently,
\begin{equation}
\label{|h_n|_3}
|h_{n; j_1,\dots,j_k}(z)|\le {m+1\choose k} C_{j_1,\dots,j_k}e^{2S\myt C_1}e^{-n(u_k(z)-u_{k-1}(z))}.
\end{equation}
The functions $u_i$ are continuous  in~$\CC$  (see Appendix~1, Lemma~1 in \cite{ChKoPaSu})  and besides,
for $k>1$, the function $u_k(z)-u_{k-1}(z)$ is continuous near~$\infty$ and for $k=1$
it converges to~$+\infty$ as $z\to\infty$. Hence, since the compact set~$K$ does not intersect with $F_k$, we have
$\varkappa:=\min\limits_{z\in K}
\left( u_k(z)-u_{k-1}(z)\right)>0$. 
So,
\begin{equation*}
|h_{n; j_1,\dots,j_k}(z)|\le {m+1\choose k} C_{j_1,\dots,j_k}e^{2S\myt C_1}e^{-n\varkappa}.
\end{equation*}
Since $\varkappa>0$, we have  $\lim\limits_{n\to\infty}\max\limits_{z\in K^\varepsilon(n)}|h_{n;j_1,\dots,j_k}(z)| = 0$.

Let us now prove \eqref{P/P^}. From \eqref{Pn/Pn} we obtain that, for $z\in K$,
\begin{equation}
\label{P/P-M/M}
\begin{aligned}
\left|\frac{P_{n;j_1,\dots,j_k}(z)}{P_{n;i_1,\dots,i_k}(z)}                                                 -\frac{M_{j_1,\dots,j_k}(\myt\zz^{(01\dots k-1)})}{M_{i_1,\dots,_k}(\myt\zz^{(01\dots k-1)})}\right|=\\
\left|\frac{M_{j_1,\dots,j_k}(\myt\zz^{(01\dots k-1)})}{M_{i_1,\dots,_k}(\myt\zz^{(01\dots k-1)})}\right|\cdot
\left|\frac{h_{n; j_1\dots,j_k}(z)-h_{n;i_1\dots i_k}(z)}{1+h_{n;i_1\dots i_k}(z)}\right|\le\\
\left|\frac{M_{j_1,\dots,j_k}(\myt\zz^{(01\dots k-1)})}{M_{i_1,\dots,_k}(\myt\zz^{(01\dots k-1)})}\right|\cdot
\frac{|h_{n; j_1\dots j_k}(z)|+|h_{n;i_1\dots i_k}(z)|}{|1+h_{n;i_1\dots i_k}(z)|}.\\
\end{aligned}
\end{equation}
We have $K_\varepsilon(n)\subset K_1$, and hence
\begin{equation}
\label{mod_M_3}
\max_{z\in K^\varepsilon(n)}\left|\frac{M_{j_1,\dots,j_k}(\myt\zz^{(01\dots k-1)})}{M_{i_1,\dots,_k}(\myt\zz^{(01\dots k-1)})}\right|\le
\max_{z\in K_1}\left|\frac{M_{j_1,\dots,j_k}(\myt\zz^{(01\dots k-1)})}{M_{i_1,\dots,_k}(\myt\zz^{(01\dots k-1)})}\right|=:
C_{j_1,\dots,j_k}^{i_1,\dots,i_k}<\infty.
\end{equation}
Since we have already shown that $\lim\limits_{n\to\infty}\max\limits_{z\in K^\varepsilon(n)}|h_{n;i_1,\dots,i_k}(z)|=0$, there exists
an~$N$  such that $|1+h_{n;i_1,\dots,i_k}(z)|>1/2$ for all $n>N$ and $z\in K^\varepsilon(n)$. So,  using \eqref{mod_M_3} and \eqref{|h_n|_3}, from \eqref{P/P-M/M} we have, for $n>N$ and $z\in K^\varepsilon(n)$,
\begin{equation}
\label{|P/P-M/M|}
\left|\frac{P_{n;j_1,\dots,j_k}(z)}{P_{n;i_1,\dots,i_k}(z)}
-\frac{M_{j_1,\dots,j_k}(\myt\zz^{(01\dots k-1)})}{M_{i_1,\dots,_k}(\myt\zz^{(01\dots k-1)})}\right|
\le \myt C_0 e^{-n(u_k(z)-u_{k-1}(z))},
\end{equation}
where $\myt C_0=2{m+1\choose k}C_{j_1,\dots,j_k}^{i_1,\dots,i_k}(C_{j_1,\dots,j_k}+C_{i_1,\dots,i_k})e^{2S\myt C_1}$ is a~constant.
Estimate \eqref{P/P^} evidently follows from \eqref{|P/P-M/M|}.
\end{proof}

\begin{Corollary}
\label{Cor_P/P}
Suppose that the surface $\TRS_{[k]}$, which is constructed from~$\pi$, is connected.
Then, for any compact set $K\subset \CC\setminus F_k$, as  $n\to\infty$ we have
\begin{equation}
\label{int_t2_1}
\frac{P_{n;j_1,\dots,j_k}(z)}{P_{n;i_1,\dots,i_k}(z)}
\xrightarrow{\mcap}
\frac{M_{j_1, \dots,j_k}(z)}{M_{i_1, \dots,i_k}(z)},
\quad z\in K.
\end{equation}
Moreover,  for an arbitrary $\varepsilon'>0$,
\begin{equation}
\mcap\left\{z\in K:\left|\frac{P_{n;j_1,\dots,j_k}(z)}{P_{n;i_1,\dots,i_k}(z)}-
\frac{M_{j_1, \dots,j_k}(z)}{M_{i_1, \dots,i_k}(z)}\right|^{1/n}\cdot
e^{u_{k}(z)-u_{k-1}(z)}\ge 1+\varepsilon'\right\}
\to0.
\end{equation}
\end{Corollary}

\begin{proof}
To prove this corollary, we will, roughly speaking, show that in Statement~\ref{St_P/P} in the case $\infty\not\in K$
one can remove small Euclidean discs, instead of discs in the spherical metric, from the compact set~$K$ and not take into account those discs whose centres are `far' from~$K$.
So, we denote by $B^\varepsilon_{z^*}:=\{z:|z-z^*|<\varepsilon\}$ the disc of radius $\varepsilon$ with  centre at~$z^*$ in the standard Euclidean metric on~$\CC$.
We set
\begin{equation}
r= \max\limits_{\zeta\in K}\dist(0,\zeta)+\dfrac13\dist(K,\infty), \quad
R= \max\limits_{\zeta\in K}\dist(0,\zeta)+\dfrac23\dist(K,\infty).
\end{equation}
Denote  by $\{v_i(n)\}_{i=1}^{l(n)}$, $l(n)\le 2S+W$ the points among $a_i(n), b_i, w_i$ that lie in $\myo {O^r_0}$, and define $\myt K^\varepsilon(n):=K\setminus\bigcup_{i=1}^{l(n)}B_{v_i(n)}^\varepsilon$.
Since the spherical metric and the Euclidean metric on $\myo {O^R_0}$ are equivalent, there exists the constant $C$  such that
$|z_1-z_2|\le C\dist(z_1,z_2)$  for $z_1,z_2\in \myo {O^R_0}$. Further, let $\varepsilon<\frac{C}{3}\dist(K,\infty)$. Then, on the one
hand, for $z^*\in \myo {O^r_0}$ we have $O^{\varepsilon/C}_{z^*}\subset B^\varepsilon_{z^*}$, and on the other hand,
for $z^*\in \myh\CC\setminus\myo {O^r_0}$ we have $O^{\varepsilon/C}_{z^*}\cap K = \varnothing$. Consequently,
$\myt K^\varepsilon(n)\subset K^{\varepsilon/C}(n)$. Hence, in the case $\infty\not\in K$ the compact sets $K^\varepsilon(n)$
can be replaced by $\myt K^\varepsilon(n)$, and the conclusion of Statement~\ref{St_P/P} remains valid. Since
$K\setminus\myt K^\varepsilon(n)$ is contained in the union of at most $2S + W$ Euclidean discs of radius $\varepsilon$ with centres
on the compact set $\myo {O^r_0}$, we have $\mcap(K\setminus\myt K^\varepsilon(n))\le \const \varepsilon^{1/(2S+W)}$ (this
follows from the standard estimate for the capacity of union of sets, see, for example, \cite{Ra}, Theorem 5.1.4).
Therefore,  $\mcap(K\setminus\myt K^\varepsilon(n))\to 0$ as $n\to \infty$.
Now the result of the corollary clearly follows from  Statement~\ref{St_P/P} (with $K^\varepsilon(n)$ replaced by  $\myt K^\varepsilon(n)$).
\end{proof}

\begin{Remark}
\label{t2_rem}
Since in Statement~\ref{St_P/P} and in Corollary~\ref{Cor_P/P} we consider ratios of $k$th polynomials
of the Hermite--Pad\'e $m$-system, they will also hold true if one considers arbitrary $k$th polynomials of the Hermite--Pad\'e $m$-system
satisfying~\eqref{khp_i} (not only those for that the function $\log|\myt R_n|$ is spherically normalized on the sheet $\TRS_{[k]}^{(01\dots k-1)}$
(see~\eqref{|tRn|_2})).
\end{Remark}

\section{The condition of connectedness of the Riemann surface $\TRS_{[k]}$}
\label{s6}

In this section, we discuss the condition from Theorems~\ref{theorem1} and \ref{theorem2} of connectedness
of the  surface $\TRS_{[k]}$.
First of all, we explain why this condition is required for the proof (by our method) of Theorems~\ref{theorem1} and~\ref{theorem2}.
The key point for us was the derivation of expression \eqref{div_tRn}
for the divisor of $(\myt R_n)$ and the deduction of representation \eqref{|tRn|} for the function $|\myt R_n|$ from this expression:
\begin{equation*}
|\myt R_n(\myt\zz)|=C_n e^{-n\myt u(\myt\zz)}\psi_n(\myt\zz),
\end{equation*}
where $C_n>0$ is a~constant.
After this, we fixed the normalization of the functions $\myt u$ and $\log\psi_n$, it was spherical on the
sheet $\myt\RS_{[k]}^{(01\dots k-1)}$~\eqref{norm_g}. Next, since $\myt R_n$
is expressed in terms of  $P_{n;i_1,\dots,i_k}$ via \eqref{tRn}, 
multiplying all $P_{n;i_1,\dots,i_k}$ by the same constant, we were able to choose an arbitrary (convenient for us) constant $C_n$
(we put  $C_n=1$~\eqref{|tRn|_2}). In the case, when the surface $\TRS_{[k]}$  is disconnected, from
\eqref{div_tRn} it follows that representation \eqref{|tRn|} holds true only on each connected component of $\TRS_{[k]}$. Therefore,  the constant $C_n$
can be distinct for each connected component,
but we can multiply all $P_{n;i_1,\dots,i_k}$, and hence $\myt R_n$,
only by one common constant, that is,  we can only simultaneously multiply all $C_n$ by the same constant.
So, in this case we can not choose an arbitrary suitable normalization of the function $\myt R_n$.
Moreover, if $\TRS_{[k]}$ is disconnected, then the condition $\myt R_n\not\equiv 0$
implies only that there exists at least one connected component on that $\myt R_n\not\equiv 0$,
but $\myt R_n$ can be identically 0 on other components. It can be shown that in this case representation
\eqref{|tRn|} for $|\myt R_n|$ holds true with its own constant $C_n\ge 0$ on each connected component  of $\TRS_{[k]}$, that is,  $C_n$
may vanish on some components.
Besides that, as $\myt R_n$, the functions $M_{j_1,\dots,j_k}(\myt\zz)$~\eqref{MM}
may identically vanish on some connected components of $\TRS_{[k]}$. Therefore, if $\TRS_{[k]}$
is disconnected,
then
in the representation for $   P_{n;i_1,\dots,i_k}$ in terms of $\myt R_n$~\eqref{P_M_1},
the principal asymptotic term may not correspond  to the sheet $\myt\RS_{[k]}^{(01\dots k-1)}$,
even if we assume that  the problem of normalization of~$\myt R_n$ is somehow solved. In particular,
it may happen that in \eqref{P_Arg} the function $M_{j_1,\dots,j_k}(\myt\zz^{(01\dots k-1)})$
vanish identically in some neighbourhood.
Moreover, if $\myt R_n(\myt\zz^{(01\dots k-1)})$ vanishes in some neighbourhood, then representation
\eqref{P_Arg} itself is incorrect, because in this case the function $h_{n; j_1,\dots,j_k}$~\eqref{h_n} is not defined.

Now let us discuss the condition of connectedness of the  surface $\TRS_{[k]}$.
Recall that we fixed the compact Riemann surface~$\RS$ and
the $(m+1)$-sheeted branched covering $\pi\colon \RS\to\myh\CC$ of~$\myh\CC$. 
So, we can consider $\RS$ as a~standard compactification of the Riemann surface $\RS'$ of the
$(m+1)$-sheeted global analytic function (GAF) $w(\cdot):=\pi^{-1}(\cdot)$ in the domain $\myh\CC\setminus\Sigma$ (where $\Sigma$
is the set of all critical values of~$\pi$).
The points of~$\RS'$ are pairs $(z, w^z)$, where $z\in\myh\CC\setminus\Sigma$ and $w^z$
is a~germ of the GAF~$w$ at the point~$z$.
The surface $\TRS_{[k]}$ is defined as the standard compactification of the Riemann surface $\TRS_{[k]}'$
of all unordered collections of~$k$ distinct germs of the GAF~$w$ that are
considered at the same points $z\in\myh\CC\setminus\Sigma$ (for more details, see \S\,\ref{s4}).
So, the connectedness of $\TRS_{[k]}$ is equivalent to  the connectedness of~$\TRS_{[k]}'$.
In our arguments, we will prove the connectedness or disconnectedness exactly of the surface $\TRS_{[k]}'$, not specifying it in formulations. As before, the points of the  surface $\TRS_{[k]}'$ will be denote by $(z, \{w_1^z,\dots,w^z_{k}\})$,
where $z\in\myh\CC\setminus\Sigma$ and $\{w^z_1,\dots,w^z_{k}\}$
is an unordered collection of $k$~distinct germs of the function  $w(\cdot)$ at the point~$z$.

First, we note that for $k=1$ the surface $\TRS_{[1]}'$ precisely coincides with the surface $\RS'=\RS\setminus\pi^{-1}(\Sigma)$ by construction.
It is easily seen that, for $k=m$, the surface $\TRS_{[m]}'$ is isomorphic to the  surface~$\RS'$.
Indeed, let $(z, \{w_1^z,\dots,w^z_{m}\})\in\TRS_{[m]}'$.
We define $w^z_{m+1}$ as the only of the germs of the GAF~$w$ at the point~$z$ that is not contained in the collection $\{w_1^z,\dots,w^z_{m}\}$.
The required isomorphism sends  $(z, \{w_1^z,\dots,w^z_{m}\})$ to $(z, w^z_{m+1})$.
So, the surfaces $\TRS_{[1]}$ and $\TRS_{[m]}$ are always connected.

Now let us show that for the following class of projections~$\pi$ all the surfaces $\TRS_{[k]}$, $k=1,\dots, m$, are connected.

\begin{Statement}
\label{St_conn}
Assume that a projection $\pi\colon \RS\to\myh\CC$ is such that all its critical points are of the first order and
that for each point $z\in\myh\CC$ there is at most one critical point of~$\pi$ over it (i.e. in the set $\pi^{-1}(z)$). Then  all the surfaces $\TRS_{[k]}$, $k=1,\dots, m$, are connected.
\end{Statement}

\begin{proof}
Since, as mention above, the  surfaces  $\TRS_{[1]}$ and $\TRS_{[m]}$ are always connected, we will assume that $m\ge 3$ and some $k=2,\dots, m-1$ is fixed.  Let us show that the surface $\TRS_{[k]}'$ is connected.

Let the set $\Sigma$ of branch points of the GAF $w:=\pi^{-1}$ consist of the points $a_1,\dots,a_J$.
The condition of the theorem means that over each point $a_j$ there is precisely one branch point of the GAF~$w$ and its order is~2.
Since the set $\Sigma$ is finite, there exists a~point $a\in\CC$ such that all closed intervals  $[a,a_j]$, $j=1,\dots,J$,
intersect only in the point $a$ (if $\infty\in\Sigma$, then by the interval $[a,\infty]$
we mean any ray from $a$ to~$\infty$ not containing other points~$a_j$). We set $S:=\bigcup_{j=1}^J [a,a_j]$. Then,  $\myh\CC\setminus S$ is connected.
It is clear that  over $\myh\CC\setminus S$ the surface $\RS$ splits into  $m+1$ disjoint connected sheets and $\pi$ is
biholomorphic on each of them. In terms of~$w$, this means that over $\myh\CC\setminus S$ the GAF~$w$ splits into  $m+1$
distinct holomorphic  functions (the branches of~$w$), which will be denote by $w_1(\cdot),\dots, w_{m+1}(\cdot)$.
In what follows, we assume that for $z\in\myh\CC\setminus S$ the germ $w_i^z$ is the germ of exactly the function $w_i(\cdot)$
at the point $z$, $i=1,\dots, m+1$.
Fix $z^*\in\myh\CC\setminus S$. We will
show that, for any tuple of distinct $k+1$ indices $i_1,\dots,i_{k+1}$, $1\le i_s\le m+1$,
there exists a~path $\gamma\subset\myh\CC\setminus\Sigma$ beginning and ending at~$z^*$
such that continuing each of the germs $w^{z^*}_{i_s}$, $s=1,\dots, k-1$, along this path, we again
get the germ $w^{z^*}_{i_s}$, and continuing $w^{z^*}_{i_k}$ we get $w^{z^*}_{i_{k+1}}$.
So, the lifting of this path to~$\TRS_{[k]}'$ connects the point
$(z^*, \{w^{z^*}_{i_1},\dots,w^{z^*}_{i_{k-1}}, w^{z^*}_{i_{k}}\})$ to the point $(z^*, \{w^{z^*}_{i_1},\dots,w^{z^*}_{i_{k-1}}, w^{z^*}_{i_{k+1}}\})$.
It is clear that taking a composition of several paths analogous to~$\gamma$
we can construct a~path connecting
the point $(z^*, \{w^{z^*}_{i_1},\dots, w^{z^*}_{i_{k}}\})$ to an arbitrary given point $(z^*, \{w^{z^*}_{j_1},\dots, w^{z^*}_{j_{k}}\})\in\TRS_{[k]}'$.
Since $\myh\CC\setminus S$ is connected, the last means that $\TRS_{[k]}'$ is also connected.

So, let us construct the path $\gamma$.
By the assumption, the surface $\RS'$ of the GAF $w$ is connected,
and hence there exists a~path $\gamma'\subset\myh\CC\setminus\Sigma$ such that the continuation of~$w^{z^*}_{i_k}$
along this path gives  $w^{z^*}_{i_{k+1}}$.
For each $j=1,\dots J-1$, we fix an~(oriented) loop~$\alpha_j$ around the point $a_j$
with beginning and end at~$z^*$ such that  $\alpha_j$ intersects~$S$ in a~unique point that lies on the interval $(a, a_j)$.
Then the path $\gamma'$, as a~path in $\myh\CC\setminus\Sigma$,
is homotopic to  the path~$\gamma''$ that consists of compositions of some paths $\alpha_j$ and $\alpha_j^{-1}$, where
$\alpha_j^{-1}$~is the loop $\alpha_j$ gone in opposite direction.
Consequently,  the results of continuation of any germ of the GAF~$w$ along the paths $\gamma'$ and~$\gamma''$
are the same; in particular, $w^{z^*}_{i_k}$ is continued along $\gamma''$ to~$w^{z^*}_{i_{k+1}}$.
Let us see what happens with the germs of the  GAF~$w$ after their continuation along the loops~$\alpha_j$.
Since all the branch points of~$w$ are of second order, the results of continuaton of germs of~$w$ at the point $z^*$ along the  loops $\alpha_j$ and $\alpha_j^{-1}$ coincide. Therefore we further assume that the path $\gamma''$ consists only of the loops~$\alpha_j$.
Since over each point~$a_j$ there is precisely one branch point of second order, we see that when we go along  the loop~$\alpha_j$, only two germs from the whole collection $w^{z^*}_1,\dots,w^{z^*}_{m+1}$
are interchanged and all other germs remain unchanged.
Suppose that, when we go along he loop~$\alpha_j$, the germ $w^{z^*}_s$ is transformed to the germ $w^{z^*}_t$ (possibly coinciding with $w^{z^*}_s$). Then we write  $w^{z^*}_s\xrightarrow{\alpha_j} w^{z^*}_t$.
So, assume that the path $\gamma''$ consists of~$B$ loops $\alpha_j$:
\begin{equation}
\label{gamma''}
\gamma''=\alpha_{j_B}\circ\dots\circ\alpha_{j_1},
\end{equation}
where $j_s\subset\{1,\dots,J-1\}$.
Set  $w^{z^*}_{r_1}:=w^{z^*}_{i_k}$, $w^{z^*}_{r_{B+1}}:=w^{z^*}_{i_{k+1}}$.
Let  $w^{z^*}_{r_{l}}\xrightarrow{\alpha_{j_l}} w^{z^*}_{r_{l+1}}$, $l=1,\dots,B$, that is,
\begin{equation}
\label{chain}
w^{z^*}_{i_k}=:w^{z^*}_{r_1}\xrightarrow{\alpha_{j_1}}w^{z^*}_{r_2}
\xrightarrow{\alpha_{j_2}}\dots\xrightarrow{\alpha_{j_{B-1}}}w^{z^*}_{r_{B}}
\xrightarrow{\alpha_{j_{B}}}
w^{z^*}_{r_{B+1}}:=w^{z^*}_{i_{k+1}}.
\end{equation}
Let $w^{z^*}_{r_{l}}=w^{z^*}_{r_{l'}}$ for some $1\le l<l'\le B+1$. Then, deleting the
piece  $\alpha_{j_{l'}}\circ\dots\circ\alpha_{j_{l}}$ from the representation for~$\gamma'' $~\eqref{gamma''},
 we again obtain a~path that transforms $w^{z^*}_{i_k}$ to $w^{z^*}_{i_{k+1}}$.
Hence in what follows we assume that $\gamma''$~is a~path for that all $w^{z^*}_{r_{l}}$ in~\eqref{chain} are distinct.
Let us show that  in this case the path
\begin{equation}
\label{gamma}
\gamma:=\alpha_{j_1}\circ\dots\circ\alpha_{j_{B-1}}\circ\alpha_{j_B}\circ\alpha_{j_{B-1}}\circ\dots\circ\alpha_{j_1}
\end{equation}
is the required one.

First we show that the continuation of $w^{z^*}_{i_k}$ along $\gamma$ gives $w^{z^*}_{i_{k+1}}$.
Since, by construction, all $w^{z^*}_{r_{l}}$ in \eqref{chain} are distinct, 
when we go along the loop $\alpha_{j_l}$, $l=1,\dots,B$,
the germs $w^{z^*}_{r_{l}}$ and $w^{z^*}_{r_{l+1}}$ are interchanged and all other germs of~$w$ at $z^*$ remain unchanged. So, when we go along the loop $\alpha_{j_l}$, $l=1,\dots,B-1$,
the germ $w^{z^*}_{i_{k+1}}$ is unchanged.
The continuation of $w^{z^*}_{i_k}$ along~$\gamma''$~\eqref{gamma''} gives $w^{z^*}_{i_{k+1}}$,
and hence, taking into account the form of~$\gamma$~\eqref{gamma},
we conclude that the continuation of $w^{z^*}_{i_k}$ along~$\gamma$ also gives $w^{z^*}_{i_{k+1}}$.
Now let us show that, for all $s=1,\dots,k-1$, the continuation of $w^{z^*}_{i_s}$ along~$\gamma$ gives
the same germ $w^{z^*}_{i_{s}}$. We fix such an~$s$. Let the continuation
of $w^{z^*}_{i_s}$ along the path~$\alpha_{j_{B-1}}\circ\dots\circ\alpha_{j_1}$ be some germ $w^{z^*}_b$.
Let us show that when we go along the loop~$\alpha_{j_B}$, the germ $w^{z^*}_b$ is unchanged. Assume the contrary.
Then, since all $w^{z^*}_{r_{l}}$ in~\eqref{chain} are distinct, we have $w^{z^*}_b = w^{z^*}_{r_{B+1}}$.
Hence  the continuation of the germ $w^{z^*}_{i_s}$ along the path $\alpha_{j_{B-1}}\circ\dots\circ\alpha_{j_1}$ gives $w^{z^*}_{r_{B+1}}=w^{z^*}_{i_{k+1}}$.
On the other hand, when we go along the loop $\alpha_{j_l}$, $l=1,\dots,B$, the germs $w^{z^*}_{r_{l}}$ and $w^{z^*}_{r_{l+1}}$ are interchanged and all other germs remain unchanged.
Consequently, since $w^{z^*}_{i_{k+1}}\ne w^{z^*}_{i_s}$, among the germs $w^{z^*}_{r_{l}}$, $l=1,\dots,B-1$,
there is the germ $w^{z^*}_{i_{k+1}}=w^{z^*}_{r_{B+1}}$. But this contradicts the  fact that
all $w^{z^*}_{r_{l}}$  in~\eqref{chain} are distinct. Hence the continuation of the germ $ w^{z^*}_{i_s}$ along the path~$\gamma$~\eqref{gamma}
coincides with its continuation along the path
\begin{equation}
\alpha_{j_1}\circ\dots\circ\alpha_{j_{B-1}}\circ\alpha_{j_{B-1}}\circ\dots\circ\alpha_{j_1}\equiv \Id,
\end{equation}
that is, this germ is unchanged under continuation along $\gamma$.
\end{proof}

Though Statement~\ref{St_conn} shows that the class of projections~$\pi$ for that the  surfaces $\TRS_{[k]}$ are connected
is quite broad, the following Proposition~\ref{Pr_disco} gives a~natural class of projections~$\pi$ for that the
surfaces $\TRS_{[k]}$ for all $k=2,\dots,m-1$ are disconnected.

Since $\RS$ is a compact Riemann surface and since $\pi\colon \RS\to\myh\CC$ is an $ (m+1)$-sheeted branched covering of $\myh\CC$, the
covering $\pi\colon \RS\to\myh\CC$ is isomorphic to the covering of $\myh\CC$ by the algebraic surface
defined in the affine part of $\myh\CC_z\times\myh\CC_w$ by some algebraic equation $P(z, w)=0$, where $P$
is an irreducible polynomial of degree $(m+1)$ with respect to~$w$, with the natural projection $(z,w)\mapsto z$. 
Respectively,
the  surface $\RS'$ is isomorphic to the Riemann surface of the GAF~$w$ in $\myh\CC\setminus\Sigma$.

\begin{Proposition}
\label{Pr_disco}
Let $\RS$ be the Riemann surface defined by the algebraic equation $w^{m+1}=R(z)$, where $m\ge 3$
and $R(z)$ is an arbitrary rational function, and let $\pi\colon \zz=(z,w)\mapsto z$. Then the  surfaces $\TRS_{[k]}$ for all $k=2,\dots,m-1$ are disconnected.
\end{Proposition}

\begin{proof}
The  surfaces $\TRS_{[k]}'$ and $\TRS_{[m+1-k]}'$ are isomorphic (the required isomorphism
sends $(z, \{w_1^z,\dots,w^z_{k}\})$ to the point $(z, \{w_{k+1}^z,\dots,w^z_{m+1}\})$, where $w_{k+1}^z,\dots$, $w^z_{m+1}$
are the germs of~$w$ at the point $z$ that are not contained among the germs $w_1^z,\dots,w^z_{k}$). So we can assume that
$k\le (m+1)/2$.

The branch points of the GAF $w(z)=\sqrt[m+1]{R(z)}$ are contained among the zeros and poles
of the function  $R(z)$. Let $a$ be a~zero or a~pole of~$R(z)$ of order~$l$,
and let $d$~be the greatest common divisor of $(m+1)$ and~$l$. Then
over~$a$ the GAF~$w$ has $d$~branch points of order $\frac{m+1}{d}$.
Therefore, when we go around the point $a$, the ratio of any two germs of the function~$w$ is preserved. 
Hence if $w^{z^*}_r, w^{z^*}_t$ are two germs of the GAF $w(z)$ at an arbitrary point $z^*\in\myh\CC\setminus \Sigma$, then
the ratio $w^{z^*}_r/ w^{z^*}_t$ is preserved under the continuation along any path $\gamma$ in $\myh\CC\setminus \Sigma$ with beginning and end at~$z^*$.

We fix a~point $z^*\in\CC$ that is neither a~zero nor a~pole of the function $R(z)$. Let $w^{z^*}_1$ be some germ of~$w$ at~$z^*$.
For each $j=2,\dots,m+1$, we denote by $w^{z^*}_j$ the germ of~$w$ at~$z^*$ such that
\begin{equation*}
w^{z^*}_j(z^*)=e^{2\pi i\frac{(j-1)}{m+1}}w^{z^*}_1(z^*).
\end{equation*}
Recall that we assume that  $k\le (m+1)/2$. Let us show that there is no path $\gamma\subset\myh\CC\setminus \Sigma$ with beginning and end at~$z^*$
such that the continuation along this path transforms the unordered collection of germs $\{w^{z^*}_1,w^{z^*}_2,\dots,w^{z^*}_{k}\}$ to the collection $\{w^{z^*}_1,w^{z^*}_2,\dots,w^{z^*}_{k-1}, w^{z^*}_{s}\}$, where $s=(m+1)/2+1$ for odd $m$ and $s=m/2+1$ for even~$m$.
Indeed, we have $\arg(w^{z^*}_s(z^*)/w^{z^*}_1(z^*))=2\pi(s-1)/(m+1)$,
and the arguments of all possible ratios $w^{z^*}_r(z^*)/w^{z^*}_t(z^*)$ for  $r,t\le k\le (m+1)/2$
are not greater than $2\pi(s-2)/(m+1)$ (modulo~$2\pi$). Therefore,  among the germs $w^{z^*}_1,w^{z^*}_2,\dots,w^{z^*}_{k}$
there are no such ones whose ratio at the point $z^*$ is $w^{z^*}_s(z^*)/w^{z^*}_1(z^*)$.
Since, according to the above, the ratio of two arbitrary germs of~$w$  is preserved under the continuation along any closed path,
the last claim proves that the required~$\gamma$ does not exist.
\end{proof}

In the end of this section we will consider in more detail the simplest case where $m=3$ and, respectively, $k=2$.
Since in this case the covering $\pi\colon \RS\to\myh\CC$ is 4-sheeted,
it follows that the covering $\myt\pi\colon\TRS_{[2]}\to\myh\CC$ is 6-sheeted, and it is easily seen that
any connected component of $\myt\pi\colon\TRS_{[2]}\to\myh\CC$ is at least 2-sheeted. So, we have three possibilities:
1)~$\myt\pi\colon\TRS_{[2]}\to\myh\CC$ consists of three 2-sheeted connected components,
2)~$\myt\pi\colon\TRS_{[2]}\to\myh\CC$ consists of one 2-sheeted and one  4-sheeted connected components,
3)~$\myt\pi\colon\TRS_{[2]}\to\myh\CC$ is a~connected 6-sheeted covering.
We will show that all these three possibilities are realized giving the explicit examples.
All the corresponding surfaces $\RS$ will be topological spheres (i.e. of zero genus).

\begin{Example}
\label{ex1}
Let $\RS$ be a~compactification  of the Riemann surface of the GAF $w(z)=\sqrt{z}+\sqrt{z-1}$ and
let $\pi\colon (z,w)\mapsto z$. It is easily seen that the covering $\myt\pi\colon\TRS_{[2]}\to\myh\CC$ is isomorphic
(under a~base preserving isomorphism) to the disjoint union of the coverings of $\myh\CC$ with the help of the GAFs $w_{(1)}(z) = \sqrt z$, $w_{(2)}(z) = \sqrt {z-1}$, $w_{(3)}(z) = \sqrt {z/(z-1)}$.
\end{Example}

\begin{Example}
\label{ex2}
Let $\RS$ be a~compactification of the Riemann surface of the GAF $w(z)=\sqrt[4]{z}$ and let $\pi\colon (z,w)\mapsto z$. It is easily seen that
the covering $\myt\pi\colon\TRS_{[2]}\to\myh\CC$ is isomorphic  (under a~base preserving isomorphism)
to the disjoint union of the coverings of $\myh\CC$ with the help of the GAFs $w_{(1)}(z) = \sqrt z$, $w_{(2)}(z) = \sqrt[4] {z}$.
\end{Example}

Examples \ref{ex1} and \ref{ex2} show that, for coverings $\pi$ defined explicitly via the GAF $w(z)=\pi^{-1}(z)$ that can be
expressed in terms of radicals of~$z$,
it is hard to expect the surface $\TRS_{[2]}$ to be connected. In addition, it is easily shown that if $w$~satisfies the
polynomial equation $P(z,w)=0$ that is biquadratic in~$w$,
then the corresponding surface $\TRS_{[2]}$ is always disconnected.

Statement~\ref{St_conn} provides a~broad class of examples where the surface $\TRS_{[2]}$ is connected.
However, statement~\ref{St_conn} itself does not explicitly define such surfaces, for example, as the solution of a corresponding equation.
Let a~GAF $w$ be defined by an algebraic equation $z=R(w):=P_4(w)/Q_2(w)$, where $P_4(w)$ and $Q_2(w)=(w-a)(w-b)$
are polynomials of degrees  4~and~2, respectively, that have no common roots and $a\ne b$. We also require that the derivative $R'(w)$ has no
multiple zeros (that is, all the zeros $w_1,\dots, w_5$ of the polynomial $P_4'Q_2-P_4Q_2'$ are distinct).
Let $\RS$ be a~compactification of the Riemann surface of this GAF~$w$ and let $\pi\colon (z,w)\mapsto z$. Then  $\RS$
is a~topological sphere  (since $w$~is a~global coordinate) and the GAF $w=\pi^{-1}$
has precisely 5~branch points of second order in the affine part $R(w_1),\dots, R(w_5)$ and
one branch point of second order at~$\infty$.  Therefore,  in order for such~$\pi$ to satisfy the
condition of Statement~\ref{St_conn}, and respectively,
the surface $\TRS_{[2]}$ to be connected, it is necessary to require that $R(w_j)\ne R(w_l)$ for $l\ne j$.
Let us give an explicit example of such a~surface.

\begin{Example}
\label{ex3}
Let $\RS$ be the Riemann surface of the algebraic function~$w$ defined as the solution of the equation
\begin{equation}
\label{ex_expl}
z=\frac{w^4-(1+i)w^3+3iw^2}{w^2+\frac{1+i}{3}w+\frac{i}{3}}
\end{equation}
and let $\pi\colon (z,w)\mapsto z$. Then  the surface $\TRS_{[2]}$ is connected.
\end{Example}
\begin{proof}
We denote by $R(w)$ the  right-hand side of~\eqref{ex_expl}, by $P_4(w)$ its numerator, and by $Q_2(w)$ its denominator. Then
\begin{equation*}
P_4'(w)Q_2(w)-P_4(w)Q_2'(w)=2w(w^4-1).
\end{equation*}
Consequently,  $0, \pm 1, \pm i$ are the zeros of $R(w)$.
In view of the above we should check that the function~$R$ has distinct values at these points. Indeed,
we have
$R(0)=0$, $R(1)=(3+6i)/5$, $ R(-1)=3+6i$, $R(i)=(-3+6i)/5$, $R(-i)=-3+6i$.
\end{proof}
%\end{fulltext}

{\bf Aleksandr V. ~Komlov}

Steklov Mathematical Institute of RAS, 

Moscow, Russia

{\it E-mail}: komlov@mi-ras.ru

\end{document}